\documentclass{article}

\usepackage[T1]{fontenc}
\usepackage[utf8]{inputenc}

\usepackage[margin=1.3in]{geometry}
\usepackage{graphicx}

\usepackage{amsmath,amssymb,amsthm}
\numberwithin{equation}{section}

\usepackage[
natbib,
maxcitenames=3, 
maxbibnames=10,  
sorting=ydnt,
doi=false,
url=false,
sortcites,
defernumbers,
backend=biber
]{biblatex}
\usepackage{hyperref}

\usepackage{todonotes}
\usepackage{xurl}

\usepackage[nameinlink, capitalise, noabbrev]{cleveref}

\usepackage{subcaption}
\usepackage{xfrac}
\usepackage{nicefrac}
\usepackage{tikz}
\usepackage{pgfplots}
\usetikzlibrary{3d,calc}

\pgfplotsset{compat=1.15}

\tikzset{
line/.style={very thick, color=black},
marker/.style={very thick, color=black,only marks,mark=*,mark size=2},
}

\pgfplotsset{emptydot/.style={marker,fill=white}} 
\pgfplotsset{soliddot/.style={marker}}

\pgfplotsset{nice axis/.style={%
        axis x line=middle,%
        axis y line=middle,%
        axis equal, %
        xlabel style={anchor=west},
        ylabel style={anchor=east},
        ticks=none, 
        xlabel={}, 
        ylabel={}, 
        }}

\usepgfplotslibrary{fillbetween}
\usetikzlibrary{arrows.meta}
\tikzset{>=latex}

\newtheorem{theorem}{Theorem}[section]
\newtheorem{proposition}[theorem]{Proposition}
\newtheorem{lemma}[theorem]{Lemma}
\newtheorem{corollary}[theorem]{Corollary}
\theoremstyle{remark}
\newtheorem{remark}[theorem]{Remark}

\theoremstyle{definition}
\newtheorem{example}[theorem]{Example}
\newtheorem{definition}[theorem]{Definition}

\numberwithin{equation}{section}

\newcommand{\N}{\mathbb{N}}

\newcommand{\R}{\mathbb{R}}
\newcommand{\dom}{{\mathrm{dom}}}
\newcommand{\closure}[1]{\overline{#1}}
\newcommand{\norm}[1]{\left\Vert #1 \right\Vert}
\newcommand{\abs}[1]{\left\vert #1 \right\vert}
\renewcommand{\H}{\mathcal{H}}
\newcommand{\X}{\mathcal{X}}

\renewcommand{\div}{\operatorname{div}}
\DeclareMathOperator{\argmin}{arg\,min}
\DeclareMathOperator{\argmax}{arg\,max}
\DeclareMathOperator{\esssup}{ess\,sup}
\newcommand{\st}{\,:\,}
\newcommand{\supp}{\operatorname{supp}}
\newcommand{\dx}{\,\mathrm{d}x}
\renewcommand{\d}{\,\mathrm{d}}

\newcommand{\sign}{\operatorname{sign}}

\newcommand{\calN}{\mathcal{N}}
\newcommand{\calM}{\mathcal{M}}
\newcommand{\calP}{\mathcal{P}}

\newcommand{\calR}{\mathcal{R}}
\newcommand{\calD}{\mathcal{D}}
\newcommand{\eps}{\varepsilon}

\newcommand{\kr}{\mathrm{KR}}

\newcommand{\sg}{\mu}
\newcommand{\dist}{\operatorname{dist}}
\newcommand{\lip}{\mathrm{Lip}}
\newcommand{\C}{\mathrm{C}}
\renewcommand{\L}{\mathrm{L}}
\newcommand{\W}{\mathrm{W}}

\newcommand{\Omegamax}{\Omega_{\max}}
\newcommand{\omegamax}{\omega_{\max}}
\newcommand{\Omegain}{\Omega_\mathrm{in}}
\newcommand{\din}{d_\mathrm{in}}
\newcommand{\inradius}{r_\Omega}
\newcommand{\highridge}{\calR_\Omega}

\newcommand{\M}{\mathcal M}

\newcommand{\grad}{\nabla}
\newcommand{\defeq}{:=}

\makeatletter
\newsavebox{\@brx}
\newcommand{\llangle}[1][]{\savebox{\@brx}{\(\m@th{#1\langle}\)}%
  \mathopen{\copy\@brx\kern-0.5\wd\@brx\usebox{\@brx}}}
\newcommand{\rrangle}[1][]{\savebox{\@brx}{\(\m@th{#1\rangle}\)}%
  \mathclose{\copy\@brx\kern-0.5\wd\@brx\usebox{\@brx}}}
\makeatother

\newcommand{\restr}{\mathbin{\vrule height 1.6ex depth 0pt width
0.13ex\vrule height 0.13ex depth 0pt width 1.3ex}}

\newcommand{\wsto}{\rightharpoonup^\ast}

\newcommand{\RI}{(-\infty,+\infty]}

\makeatletter
\let\blx@rerun@biber\relax
\makeatother

\begin{document}

\title{Eigenvalue Problems in $\mathrm{L}^\infty$: Optimality Conditions, Duality, and Relations with Optimal Transport}
\author{Leon Bungert\thanks{Hausdorff Center for Mathematics, University of Bonn, Endenicher Allee 62, Villa Maria, 53115 Bonn, Germany. \href{mailto:leon.bungert@hcm.uni-bonn.de}{leon.bungert@hcm.uni-bonn.de}} 
\and Yury Korolev\thanks{Department of Mathematical Sciences, University of Bath, Claverton Down BA2 7AY, UK. \href{mailto:ymk30@bath.ac.uk}{ymk30@bath.ac.uk}}}
\date{}
\maketitle

\begin{abstract}
In this article we characterize the $\mathrm{L}^\infty$ eigenvalue problem associated to the Rayleigh quotient $\left.{\|\nabla u\|_{\mathrm{L}^\infty}}\middle/{\|u\|_\infty}\right.$ and relate it to a divergence-form PDE, similarly to what is known for $\mathrm{L}^p$ eigenvalue problems and the $p$-Laplacian for $p<\infty$. Contrary to existing methods, which study $\mathrm{L}^\infty$-problems as limits of $\mathrm{L}^p$-problems for $p\to\infty$, we develop a novel framework for analyzing the limiting problem directly using  convex analysis and geometric measure theory. 
For this, we derive a novel fine characterization of the subdifferential of the Lipschitz-constant-functional $u\mapsto\|\nabla u\|_{\mathrm{L}^\infty}$.
{We show that the eigenvalue problem takes the form $\lambda \nu u =-\operatorname{div}(\tau\nabla_\tau u)$, where $\nu$ and $\tau$ are non-negative measures concentrated where $|u|$ respectively $|\nabla u|$ are maximal, and $\nabla_\tau u$ is the tangential gradient of $u$ with respect to $\tau$.
Lastly, we investigate a dual Rayleigh quotient whose minimizers solve an optimal transport problem associated to a generalized Kantorovich--Rubinstein norm.
Our results apply to all stationary points} of the Rayleigh quotient, including infinity ground states, infinity harmonic potentials, distance functions, etc., and generalize known results in the literature.


\end{abstract}

\tableofcontents

\section{Introduction}

\subsection{{Motivation and Main Contributions}}
{
Nonlinear eigenvalue problems for the $p$-Laplacian for $p<\infty$ have been the subject of extensive research for the last three decades---see \cite{garcia1987existence,lindqvist1990equation,le2006eigenvalue,kawohl2003isoperimetric,kawohl2017geometry,kawohl2006positive,kawohl2007dirichlet,kawohl2008p} for a non-exhaustive list---and have applications in data science \cite{buhler2009spectral,hein2010inverse}.
They can be  characterized as solutions of a nonlinear divergence-form PDE or as unique minimizers of a Rayleigh quotient involving the $p$-Dirichlet energy (we refer to this as the $\L^p$ eigenvalue problem). 
For $p=\infty$ minimizers of the Rayleigh quotient, now involving the Lipschitz constant, are no longer unique. We refer to this problem as the $\L^\infty$ eigenvalue problem. 
A certain class of minimizers, called infinity ground states, can be recovered as limits of $p$-Laplacian eigenfunctions as $p\to\infty$. General minimizers, however, do not admit such a variational principle. 

In this paper we develop a novel analytical framework for studying the $\L^\infty$ eigenvalue problem which does not require taking the limit $p\to\infty$ and instead uses techniques from convex analysis and geometric measure theory. This allows us to generalize various known results about special classes of minimizers and extend them to general minimizers.

Let us fix some notation. Let $\Omega\subset\R^n$ be a bounded domain.
For $p\in[1,\infty]$, we denote the $\L^p$-spaces with respect to a measure $\mu$ as $\L^p_\mu(\Omega)$ and we write simply $\L^p(\Omega)$ when $\mu$ is the Lebesgue measure.  
These spaces are equipped with  standard $\L^p$-norms $\norm{u}_{\L^p_\mu}$ or $\norm{u}_{\L^p}$, where we omit the dependency on $\Omega$ for the sake of a compact notation. 
For $p\in[1,\infty)$, the Sobolev space $\W^{1,p}_0(\Omega)$ is defined as the closure of the space of smooth and compactly supported functions with respect to the norm $\norm{u}_{\W^{1,p}}:=\norm{u}_{\L^p}+\norm{\grad u}_{\L^p}$.

The eigenvalue problem of the $p$-Laplacian {(see~\cite{lindqvist1990equation} for a detailed study)} consists in finding a function $u\in\W^{1,p}_0(\Omega)$ which is a weak solution of
\begin{align}\label{eq:p-groundstate}
    \lambda_p^p \abs{u}^{p-2}u=-\div(\abs{\grad u}^{p-2}\grad u).
\end{align}}
The eigenvalue $\lambda_p>0$ is given by the minimal value of a nonlinear Rayleigh quotient
\begin{align}\label{eq:p-eigenvalue}
    \lambda_p\defeq \inf_{u\in\W^{1,p}_0(\Omega)}\frac{\norm{\nabla u}_{\L^p}}{\norm{u}_{\L^p}}.
\end{align}
Solutions of the $p$-Laplacian eigenvalue problem \labelcref{eq:p-groundstate} for $p>1$ are known to be unique {up to normalization} and are in one-to-one correspondence with minimizers of the Rayleigh quotient in~\labelcref{eq:p-eigenvalue}.

{In this paper we study the following limiting minimization problem of an $\L^\infty$ Rayleigh quotient over $\lip_0(\Omega)$, the space of Lipschitz functions on $\Omega$ which are zero on the boundary:
\begin{align}\label{eq:infty-eigenvalue}
    \lambda_\infty\defeq \inf_{u\in\lip_0(\Omega)}\frac{\norm{\nabla u}_{\L^\infty}}{\norm{u}_{\L^\infty}}.
\end{align}
We denote by $\inradius>0$  the inradius of $\Omega$, defined as maximal value of the distance function:
\begin{align}
    \label{eq:distance_fct}
    d_\Omega(x) &\defeq  \dist(x,\partial\Omega) \defeq  \min_{y\in\partial\Omega}\abs{x-y},\\
    \label{eq:inradius}
    \inradius&\defeq \max_{x\in\Omega}d_\Omega(x).
\end{align}
It is very easy to show \cite{juutinen1999eigenvalue} that the infimal value in \labelcref{eq:infty-eigenvalue} is given by
\begin{align}\label{eq:infty-eigenvalue_inradius}
\lambda_\infty = \lim_{p\to\infty}\lambda_p = \frac{1}{\inradius},
\end{align}
which implies that the distance function is always a minimizer of the Rayleigh quotient. 

It has been shown in~\cite{champion2008infty,katzourakis2021generalised,evans2005various} that certain classes of minimizers of~\eqref{eq:infty-eigenvalue} satisfy a divergence-form PDE which is structurally similar to~\eqref{eq:p-groundstate}.
Furthermore, a connection between infinity ground states  and solutions of a certain optimal transport problem was established in  \cite{champion2008infty}.

In this paper we ask  the following questions:
\begin{enumerate}
    \item Do all solutions of the nonlinear eigenvalue problem associated to the Rayleigh quotient ${\norm{\grad u}_{\L^\infty}}/{\norm{u}_{\infty}}$ satisfy a PDE which is structurally similar to the $p$-Laplacian eigenvalue problem $\lambda \abs{u}^{p-2}u=-\div(\abs{\grad u}^{p-2}\grad u)$?
    \item Can all minimizers of the Rayleigh quotient ${\norm{\grad u}_{\L^\infty}}/{\norm{u}_{\infty}}$ be related to solutions of an optimal transport problem?
\end{enumerate}
The short answer is yes, see the PDE \labelcref{eq:divergence_PDE} and the optimal transport problem \labelcref{eq:nu_OT_problem} below.
To answer these questions we work with general \emph{stationary points} instead of minimizers of the Rayleigh quotient ${\norm{\grad u}_{\L^\infty}}/{\norm{u}_{\infty}}$ for which we derive a nonlinear eigenvalue problem in the form of a divergence PDE. 
Then we shall study minimizers of the Rayleigh quotient which we will relate to the distance function and solutions of an optimal transport problem.

The techniques we use to study the $\L^\infty$ eigenvalue problem are also novel: instead of approximating the $\L^\infty$-problem with $\L^p$-problems and sending $p$ to infinity, we mainly rely on elegant and well-established methods of convex analysis. 
On the one hand, this establishes a new analytical framework to tackle $\L^\infty$-type problems without using viscosity solutions or similar technical concepts from PDE analysis. 
On the other hand, this makes our results more general since the class of minimizers to the $\L^\infty$-problem considered is strictly larger than the class of minimizers which can be approximated with $\L^p$-problems.

Our \textbf{main contributions} are the following:
\begin{itemize}
    \item We develop a novel analytical framework solely based on convex analysis and geometric measure theory which allows us to prove known and novel results for $\L^\infty$-problems without the need to take the technical limit $p\to\infty$.
    \item We derive a nonlinear eigenvalue problem, involving duality maps and subdifferentials, which describes stationary points of the Rayleigh quotient ${\norm{\nabla u}_{\L^\infty}}/{\norm{u}_\infty}$.
    \item We characterize solutions to the eigenvalue problem as solutions to a fully nonlinear PDE in divergence form,
    \begin{align}\label{eq:divergence_PDE}
        \lambda \nu u = -\div(\tau \grad_\tau u),
    \end{align}
    involving non-negative measures $\nu$ and $\tau$ which are concentrated where $|u|$ respectively $|\nabla u|$ are maximal, {and the notion of a \emph{tangential gradient} $\grad_\tau u$ developed in~\cite{bouchitte1997energies}, see also \cite{jimenez:2008, champion2008infty,luvcic2021characterisation}}.
    This is our main result \cref{thm:main}.
    \item We show geometric relations between general minimizers of the Rayleigh quotient, the distance functions to the boundary, and the distance function to a generalized inball.
    \item We derive a dual Rayleigh quotient defined on the space of measures on $\Omega$ and relate it to an optimal transport problem involving a variant of the Kantorovich--Rubinstein norm.
    In particular, \cref{prop:OT_problem} shows that the measure $\mu := \nu \norm{u}_\infty$ solves
    \begin{align}\label{eq:nu_OT_problem}
        \max_{\tilde\mu\in\calP(\Omega)}\inf_{\rho\in\calP(\partial\Omega)} W^1(\tilde\mu,\rho),
    \end{align}
    where $\calP(\Omega)$ and $\calP(\partial\Omega)$ are the spaces of probability measures on $\Omega$ and its boundary $\partial\Omega$, respectively, and $W^1(\cdot,\cdot)$ is the geodesic $1$-Wasserstein distance.
\end{itemize}}

The rest of the paper is organized as follows:
{\cref{sec:special_solutions} discusses special classes of minimizers of the $\L^\infty$ Rayleigh quotient, namely infinity ground states and infinity harmonic potentials.}
In \cref{sec:eigenvalue_banach} we introduce essential concepts from convex analysis and derive general relations of nonlinear eigenvalue problems and Rayleigh quotients {on Banach spaces}.
In particular, we show equivalence between minimizers of the Rayleigh quotient and those of a dual Rayleigh quotient, which is a new result in its own right. 
In \cref{sec:funct_measures} we define {suitable} spaces of {continuous} functions and measures {and their duality relations}.
\cref{sec:main} constitutes the core of our article where we first state our main result and some corollaries, characterizing the $\L^\infty$ eigenvalue problem and minimizers of the Rayleigh quotient, and then characterize subdifferentials to prove the result.
In \cref{sec:distance_fct} we provide some geometric relations between minimizers of the Rayleigh quotient and the distance function.
\cref{sec:ot}, where we investigate a dual Rayleigh quotient and provide an optimal transport characterization of the subgradients of minimizers using Kantorovich--Rubinstein theory, is self-contained and does not utilize the subdifferential characterizations from \cref{sec:main}.
\cref{sec:outlook} concludes the paper with a summary of our results and some open questions.

\subsection{{Special Solutions of the Eigenvalue Problem}}
\label{sec:special_solutions}

{Besides the distance function \labelcref{eq:distance_fct}, which is always a minimizer of the Rayleigh quotient in \labelcref{eq:infty-eigenvalue}, there are two other important classes of minimizers: infinity ground states and infinity harmonic potentials. 
Unless for very specific domains \cite{yu2007some}, these three different classes of minimizers are different.}

{In \cite{juutinen1999eigenvalue} it was shown that in the limit $p\to\infty$} normalized eigenfunctions $u_p\in\W^{1,p}_0(\Omega)$ of the $p$-Laplacian, i.e, solutions of \labelcref{eq:p-groundstate} with $\norm{u_p}_{\L^p}=1$, converge (up to a subsequence) uniformly to a {continuous function $u_\infty$, termed infinity ground state.
Furthermore,} $u_\infty$ is a viscosity solution of {the following PDE, which is structurally completely different from~\labelcref{eq:p-groundstate}:}
\begin{align}\label{eq:infty-groundstate}
    \min(\abs{\nabla u}-\lambda_\infty u,-\Delta_\infty u)=0.
\end{align}
Here $\lambda_\infty$ is given by the reciprocal inradius as in \labelcref{eq:infty-eigenvalue_inradius}, and $\Delta_\infty u:=\langle\grad u, D^2 u\grad u\rangle$ denotes the infinity Laplacian operator, see the seminal work \cite{aronsson2004tour} for a detailed study and \cite{juutinen1999infinity} for intriguing properties.

While every solution to \labelcref{eq:infty-groundstate} is a minimizer of the Rayleigh quotient in \labelcref{eq:infty-eigenvalue}, the converse is not true and there are typically many minimizers which do not solve \labelcref{eq:infty-groundstate}. 
Furthermore, this PDE can have solutions which do not arise as limits of solutions of \eqref{eq:p-groundstate} for $p\to\infty$ and are hence called non-variational ground states, see \cite{hynd2013nonuniqueness} for an example.
Only for a very specific class of domains $\Omega$, namely stadium-like sets as classified in \cite{yu2007some}, these ambiguities do not occur and the distance function is the unique minimizer of the Rayleigh quotient and viscosity solution of the PDE.

Apart from the distance function {and infinity ground states}, another class of minimizers of the Rayleigh quotient are infinity harmonic potentials, defined as solutions to
\begin{align}\label{eq:infinity_harmonic}
    \begin{cases}
    \Delta_\infty u = 0, \qquad & \text{in }\Omega\setminus\highridge,\\
    u = \inradius,\qquad & \text{on }\highridge\\
    u = 0, \qquad & \text{on }\partial\Omega.
    \end{cases}
\end{align}
The set $\highridge\subset\Omega$ is the so-called high ridge of $\Omega$, defined as the set of all points with maximal distance to the boundary:
\begin{align}\label{eq:highridge}
    \highridge&\defeq \argmax_{x\in\Omega}d_\Omega(x).
\end{align}
Also infinity harmonic potentials are in general no infinity ground states; a counterexample on a convex domain can be found in \cite{lindgren2012inftyharmonic}.
For interesting properties of these potentials and their streamlines we refer to \cite{lindgren2021gradient}.

\subsection{Nonlinear Eigenvalue Problems on Banach Spaces}
\label{sec:eigenvalue_banach}

Before we specialize the discussion to $\L^\infty$ eigenvalue problems, this section contains a short primer on nonlinear eigenvalue problems in Banach spaces.
We introduce some important concepts from convex analysis, e.g., subdifferentials and duality maps, introduce nonlinear eigenvalue problems, and discuss their dual versions.
The presentation follows the lines of \cite{Bungert2020,bungert2021gradient}.

We let $\X$ be a Banach space over $\R$ with topological dual space $\X^*$.
The duality product is denoted by $\langle\cdot,\cdot\rangle$ and the norm on $\X^*$ is given by
\begin{align}
    \norm{\sg}_{\X^*}\defeq \sup_{\norm{u}_\X=1}\langle\sg,u\rangle.
\end{align}
\begin{definition}[Subdifferential]
Given a convex functional $J:\X\to\RI$, the subdifferential of $J$ is defined as 
\begin{align}
    \partial J(u)=\{\sg\in\X^*\st J(u)+\langle \sg,v-u\rangle\leq J(v),\;\forall v\in\X\},\quad u\in\X.
\end{align}
\end{definition}
The subdifferential is a generalization of the Frechet derivative for non-differentiable convex functionals.
Geometrically, $\partial J(u)$ contains all slopes such that the linerarization of $J$ in $u$ with this slope lies below the graph of $J$.
By definition, $\partial J(u)$ is a subset of the dual space $\X^*$.

In the context of nonlinear eigenvalue problems, absolutely homogeneous functionals have particular importance since they can be used to formulate a plethora of eigenvalue problems, e.g., associated to linear operators, or nonlinear differential operators like the $p$-Laplacian or the porous medium operator (see, e.g., \cite{bungert2019asymptotic,hynd2017approximation,bungert2019nonlinear}).
\begin{definition}[Absolutely one-homogeneous functionals]
A functional $J:\X\to\RI$ is called \emph{absolutely one-homogeneous}, if
\begin{align}
    J(cu)=|c|J(u),\quad\forall c\in\R,\;u\in\dom(J).
\end{align}
\end{definition}
Since absolutely one-homogeneous functionals are semi-norms on subspaces of $\X$, their subdifferential can be characterized as \cite{benning2013ground,burger2016spectral}
\begin{align}\label{eq:subdiff_1-hom}
    \partial J(u)=\{\sg\in\X^*\st \langle\sg,v\rangle\leq J(v),\;\forall v\in\X,\;\langle\sg,u\rangle=J(u)\}.
\end{align}
For the specific choice $J(\cdot)=\norm{\cdot}_\X$, the subdifferential is better know as duality map, defined as follows:
\begin{definition}[Duality map]
The duality map $\Phi_\X$ of $\X$ is given by
\begin{align}\label{eq:duality_map}
    \Phi_\X(u)=\{\sg\in\X^*\st\norm{\sg}_{\X^*} \leq 1,\;\langle\sg,u\rangle=\norm{u}_\X\},\quad u\in\X.
\end{align}
By the Hahn--Banach theorem $\Phi_\X(u)$ is non-empty for any $u\in\X$.
\end{definition}

We assume without loss of generality that 
\begin{align}
    \calN(J)=\{u\in\X\st J(u)=0\}=\{0\},
\end{align}
which can always be achieved by replacing $\X$ with the quotient space $\X/\calN(J)$, see \cite{bungert2019asymptotic}.
Then we can define a nonlinear Rayleigh quotient
\begin{align}\label{eq:rayleigh}
    R(u)=\frac{J(u)}{\norm{u}_\X},\quad u\in\X\setminus\{0\}
\end{align}
and the minimal value of the Rayleigh quotient is defined as
\begin{align}
    \lambda_{\min}\defeq \inf_{u\in\X\setminus\{0\}}R(u).
\end{align}
Positivity of $\lambda_{\min}$ is equivalent to $J$ being coercive, meaning that there exists $C>0$ such that
\begin{align}\label{eq:coercive}
    C\norm{u}_\X\leq J(u),\quad\forall u\in\X.
\end{align}
In this case, obviously $C=\lambda_{\min}$ is the optimal constant in \labelcref{eq:coercive}. 

Indeed, the minimal value $\lambda_{\min}$ of the Rayleigh quotient can be interpreted as smallest eigenvalue. 
To see this we define a doubly nonlinear eigenvalue problem as follows:
\begin{definition}[Nonlinear eigenvalue problem]
We call $u\in\X$ an eigenvector with eigenvalue $\lambda\in\R$ if 
\begin{align}\label{eq:eigenvalue_prob}
    \lambda\Phi_\X(u)\cap\partial J(u)\neq\emptyset.
\end{align}
\end{definition}
The following proposition---the proof of which is standard and can be found in \cite{benning2013ground} or \cite{bungert2021gradient} in large generality---states that minimizers of $R$ coincide with eigenfunctions with eigenvalue~$\lambda_{\min}$.
\begin{proposition}\label{prop:ground_states}
It holds that $u\in\X$ minimizes $R(u)=J(u)/\norm{u}_\X$ if and only if it satisfies \labelcref{eq:eigenvalue_prob} with $\lambda\defeq \lambda_{\min}$.
Such $u\in\X$ are called \emph{ground states}.
\end{proposition}

\begin{example}[$p$-Laplacian eigenvalue problem]
Letting $\X=\L^p(\Omega)$ and $J(u)=\norm{\nabla u}_{\L^p}$ if $u\in \W_0^{1,p}(\Omega)$ the eigenvalue problem \labelcref{eq:eigenvalue_prob} is equivalent to the $p$-Laplacian eigenvalue problem
\begin{align*}
    \lambda\abs{u}^{p-2}u = -\Delta_p u.
\end{align*}
\end{example}
We conclude this section with a study of the dual eigenvalue problem to~\labelcref{eq:eigenvalue_prob}. 
For this, we define the \emph{dual functional} of $J$---not to be confused with the convex conjugate---as follows:
\begin{definition}[Dual functional]\label{def:dual_fctl}
Let $J:\X\to\RI$ be absolutely one-homogeneous. 
Then the dual functional $J_*:\X^*\to\RI$ is defined as
\begin{align}
    J_*(\sg) = \sup_{J(u)=1}\langle\sg,u\rangle,\quad\sg\in\X^*.
\end{align}
\end{definition}
Since $J$ is a semi-norm when being absolutely one-homogeneous, the dual functional is nothing but the dual semi-norm, see \cite{bungert2019nonlinear}.
In particular, it is also absolutely one-homogeneous and we can define the dual Rayleigh quotient
\begin{align}
    R_*(\sg) = \frac{\norm{\sg}_{\X^*}}{J_*(\sg)},\quad \sg\in\X^*\setminus\{0\}
\end{align}
with associated dual eigenvalue problem
\begin{align}\label{eq:eigenvalue_prob_dual}
    \lambda \partial J_*(\sg) \cap \Phi_{\X^*}(\sg) \neq \emptyset.
\end{align}
The relation to the primal Rayleigh quotient $R(u)$ and the eigenvalue problem \labelcref{eq:eigenvalue_prob} becomes clear in the following proposition, which states that a solution of the primal problem gives rise to a dual solution.
\begin{proposition}\label{prop:dual_quotient}
It holds that
\begin{align}
    \inf_{u\in\X}R(u) \leq \inf_{\sg\in\X^*}R_*(\sg),
\end{align}
with equality if the left problem admits a minimizer.
If furthermore $u\in\X$ solves \labelcref{eq:eigenvalue_prob} with $\lambda=\lambda_{\min}$, then any $\sg\in\Phi_\X(u)$ with $\lambda_{\min}\sg\in\partial J(u)$ solves~\labelcref{eq:eigenvalue_prob_dual} with $\lambda=\lambda_{\min}$.
\end{proposition}
\begin{proof}
Letting $\lambda_{\min}=\inf_u R(u)$ it holds $\lambda_{\min} \norm{u}\leq J(u)$ for all $u\in\X$. 
This implies
\begin{align*}
    J_*(\sg)=\sup_{u\in\X}\frac{\langle\sg,u\rangle}{J(u)}\leq\frac{1}{\lambda_{\min}}\sup_{u\in\X}\frac{\langle\sg,u\rangle}{\norm{u}_\X}=\frac{1}{\lambda_{\min}}\norm{\sg}_{\X^*}
\end{align*}
and hence $\lambda_{\min}\leq \inf_{\sg\in\X^*} R_*(\sg)$.

On the other hand, letting $u\neq 0$ such that $R(u)=\lambda_{\min}$ and $\sg\in\X^*$ such that $\sg\in\Phi_\X(u)$ and $\lambda_{\min}\sg\in\partial J(u)$ we obtain $\norm{\sg}_{\X^*} = 1$ and hence
\begin{align*}
    \inf_{\sg\in\X^*}R_*(\sg)\leq \frac{1}{J_*(\sg)}=\frac{1}{\sup_{u\in\X}\frac{\langle\sg,u\rangle}{J(u)}}\leq \frac{\lambda_{\min}}{\frac{\langle\lambda_{\min}\sg,u\rangle}{J(u)}}=\lambda_{\min}.
\end{align*}
Hence, we have shown $\inf_\sg R_*(\sg)=\lambda_{\min}$ and that $\sg$ is a minimizer of $R_*$. 
Showing that this implies \labelcref{eq:eigenvalue_prob_dual} with $\lambda=\lambda_{\min}$ works just as in the proof of \cref{prop:ground_states}.
\end{proof}
\begin{remark}[Reflexive spaces] 
If $\X$ is reflexive it is easy to see that the dual-dual functional $(J_*)_{*}$ equals $J$ and the same holds for the quotients $(R_*)_{*}=R$. Hence, in this case the eigenvalue problems~\labelcref{eq:eigenvalue_prob} and~\labelcref{eq:eigenvalue_prob_dual} are equivalent in the sense that the subgradients of one problem are solutions to the other problem.
\end{remark}

\subsection{Functions and Measures}
\label{sec:funct_measures}

Having some abstract theory of nonlinear eigenvalue problems in Banach spaces at hand, we now introduce the setup for the $\L^\infty$-type problem that we are studying.

For a bounded domain $\Omega\subset\R^n$  we let $\C_0({\Omega})$ denote the space of all continuous functions on ${\Omega}$ which vanish on $\partial\Omega$.
Equipped with the norm $\norm{u}_\infty \defeq \max_{{\Omega}}|u|$ this is a Banach space.  
We note that $\C_0(\Omega)=\C_0(\closure\Omega)$ and is hence a closed subspace of $\C(\closure\Omega)$.
Its dual space is given by the space of {finite and signed} Radon measures $\calM(\Omega)$ on $\Omega$ equipped with the total variation norm $\norm{\mu}_{\calM(\Omega)}\defeq\abs{\mu}(\Omega)$, and the duality pairing is
\begin{align}
    \langle\mu,u\rangle \defeq  \int_\Omega u\d\mu.
\end{align}
Weak$^*$ convergence of measures $\{\mu_n\}_{n\in\N}\subset\calM(\Omega)$ to $\mu\in\calM(\Omega)$ is denoted by $\mu_n\wsto\mu$ and means $\langle u,\mu_n\rangle\to\langle u,\mu\rangle$ for all $u\in \C_0({\Omega})$.
Because of the weak$^*$ lower semicontinuity of the total variation, one has $\abs{\mu}({\Omega})\leq \liminf_{n\to\infty}\abs{\mu_n}({\Omega})$.

{We denote non-negative measures by $\calM_+(\Omega)$ and abbreviate by $\calP(\Omega)$ the set of probability measures} which consists of all measures $\mu\in\calM_+(\Omega)$ with $\mu(\Omega)=1$.

The space of vector-valued Radon measures on $\Omega$ is denoted as $\calM(\Omega,\R^n)$ and can be equipped with the same notion of convergence.
The so-called divergence-measure fields \cite{silhavy2005divergence,chen2001theory,chen2003extended,chen2005perimeter} constitute an important subclass of vector-valued Radon measures, which will turn out to be essential for studying $\L^\infty$ variational problems.
\begin{definition}[Divergence-measure field]\label{def:divergence-measure}
A measure $\sigma\in\calM(\Omega,\R^n)$ is said to be a divergence-measure field if there is a measure $\mu\in\calM(\Omega)$ such that
\begin{align}
    -\int_\Omega\nabla\varphi\cdot\d\sigma =\int_\Omega\varphi\d\mu,\quad\forall\varphi\in \C^\infty_c(\Omega).
\end{align}
In this case we write $\div\sigma\defeq\mu$ and $\sigma\in\calD\calM(\Omega,\R^n)$.
\end{definition}

We let $\lip(\closure{\Omega})$ denote the space of all Lipschitz continuous functions on $\closure{\Omega}$ and let $\lip_0(\closure{\Omega})$ be the subspace of Lipschitz-functions vanishing on $\partial\Omega$. 
A norm on $\lip(\closure{\Omega})$ is given by $\norm{u}_{\lip(\closure{\Omega})}\defeq \max(\norm{u}_\infty,\lip(u))$, where $\lip(u)$ denotes the Lipschitz constant of $u\in\lip(\closure{\Omega})$.
An equivalent norm on $\lip_0(\closure{\Omega})$ is given by $\norm{u}_{\lip_0({\Omega})}\defeq \lip(u)$ and because of the homogeneous boundary conditions it holds $\lip(u)=\norm{\nabla u}_{\L^\infty}\defeq \esssup_{x\in\Omega}\abs{\nabla u(x)}$. 
We will again simplify our notation using $\lip_0(\Omega)=\lip_0(\closure\Omega)$ and write $\lip_0(\Omega)$ throughout the paper.

Finally, the space of smooth functions on $\Omega$ is denoted by $\C^\infty(\Omega)$ and the subspace of compactly supported test functions by $\C^\infty_c(\Omega)$.

\section{Characterization of the \texorpdfstring{$\L^\infty$}{L-infinity} Eigenvalue Problem}
\label{sec:main}

For a rigorous study of {stationary points of the Rayleigh quotient}
\begin{align}\label{eq:rayleigh_inf_naive}
    {\lip_0(\Omega)\ni u\mapsto}\frac{\norm{\nabla u}_{\L^\infty}}{\norm{u}_{\L^\infty}}
\end{align}
we have to extend the functional $u\mapsto\norm{\nabla u}_{\L^\infty}$ to the Banach space $\C_0(\Omega)$ which lets us apply the abstract results of \cref{sec:eigenvalue_banach}.
For $u\in \C_0(\Omega)$ we therefore define the absolutely one-homogeneous and convex functional
\begin{align}
    J^\infty(u)=\sup\left\lbrace \int_\Omega u\div\sigma \dx\st \sigma\in \C^\infty({\Omega},\R^n),\;\norm{\sigma}_{\L^1}\leq 1 \right\rbrace,\quad u\in \C_0(\Omega),
\end{align}
which satisfies $\dom(J^\infty)= \lip_0(\Omega)$ and can be expressed as
\begin{align*}
    J^\infty(u)=
    \begin{cases}
        \norm{\nabla u}_{\L^\infty},\quad&u\in \lip_0(\Omega),\\
        +\infty,\quad&\text{else}.
    \end{cases}
\end{align*}
Therefore, the Rayleigh quotient \labelcref{eq:rayleigh_inf_naive} {can be replaced by} the Rayleigh quotient
\begin{align}\label{eq:rayleigh_inf}
    R^\infty(u) = \frac{J^\infty(u)}{\norm{u}_\infty},\qquad {u\in \C_0(\Omega).}
\end{align}
The associated abstract eigenvalue problem \labelcref{eq:eigenvalue_prob} becomes
\begin{align}\label{eq:eigenvalue_prob_inf}
    \lambda\Phi_{\C_0(\Omega)}(u)\cap\partial J^\infty(u)\neq\emptyset,
\end{align}
where $\partial J^\infty$ denotes the subdifferential \labelcref{eq:subdiff_1-hom} of $J^\infty$ with respect to $\C_0(\Omega)$.
{In particular, by} \cref{prop:ground_states} minimizers of the Rayleigh quotient \labelcref{eq:rayleigh_inf} are in one-to-one correspondence to solutions of the eigenvalue problem \labelcref{eq:eigenvalue_prob_inf} with eigenvalue $\lambda=\lambda_\infty$.

\subsection{Main Result}

{%
Here we state our main theorem, the proof of which is given at the end of \cref{sec:subdiffs}.
It features some objects which will be rigorously defined later.
For convenience, we outline their meaning here: 

The sets $\omegamax(u)$ and $\Omegamax(u)$ contain the points in $\Omega$ where $\abs{u}$ and $\abs{\nabla u}$ are maximal (in a generalized sense) and will be defined rigorously in \cref{def:omegamax,def:Omegamax}. 
The notions of the tangential gradient $\nabla_{{\tau}}u$ and \v{S}ilhav{\'y}'s pairing measure $\llangle\grad u,\sigma\rrangle$ will be introduced in \cref{def:tangetial-grad,prop:pairing}, respectively.
For now, the tangential gradient $\nabla_{{\tau}}u(x)$ can be thought of as projection of $\nabla u(x)$ onto a linear subspace ``tangential'' to the support of the measure ${\tau}$ at $x$.
Furthermore, the pairing measure $\llangle\grad u,\sigma\rrangle$ coincides with the measure $\langle\grad u, \sigma\rangle$ in case that $\grad u$ is continuous.
}

\begin{theorem}[$\L^\infty$ eigenvalue problem]\label{thm:main}
A function $u\in\C_0(\Omega)\setminus\{0\}$ solves the eigenvalue problem 
\begin{align*}
    \lambda\Phi_{\C_0(\Omega)}(u)\cap\partial J^\infty(u)\neq\emptyset
\end{align*}
if and only if there exist non-negative measures $\nu,\tau\in\calM_+(\Omega)$ such that
\begin{align}\label{eq:eigenfunction_div}
    \lambda \nu u = -\div(\tau\grad_{{\tau}}u)
\end{align}
and they have the following properties:
\begin{itemize}
    \item The measures have mass $\nu(\Omega)=\frac{1}{\norm{u}_\infty}$ and ${\tau}(\Omega)=\frac{1}{\norm{\grad u}_{\L^\infty}}$.
    \item $\nu$ is concentrated on the set $\omegamax(u)$ where $\abs{u}$ is maximal:
    \begin{align*}
        {\nu}(\Omega\setminus\omegamax(u)) = 0.
    \end{align*}
    \item $\tau$ is concentrated on  the set  $\Omegamax(u)$ where $\abs{\nabla u}$ is maximal:
    \begin{align*}
        {\tau}(\Omega\setminus\Omegamax(u))=0.
    \end{align*}
    \item The measure $\sigma:=\tau\grad_{{\tau}}u$ fulfills the following identity for \v{S}ilhav{\'y}'s pairing measure~\cite{vsilhavy2008normal}:
    \begin{align*}
    \llangle\nabla u, \sigma\rrangle &= \norm{\nabla u}_{\L^\infty} \abs{\sigma}.
    \end{align*}
\end{itemize}
\end{theorem}
Comparing the PDE \labelcref{eq:eigenfunction_div} with the $p$-Laplacian eigenvalue problem \labelcref{eq:p-groundstate} one identifies strong structural similarities.
The singular terms $\abs{u}^{p-2}$ and $\abs{\grad u}^{p-2}$ are replaced by the measures $\nu$ and $\tau$, respectively.
The occurrence of the tangential gradient $\grad_\tau u$ instead of $\grad u$ is due to the fact that our results apply to general solutions of the $\L^\infty$-problem and not only infinity harmonic limits of $\L^p$-problems, where we conjecture that $\grad_\tau u$ coincides with $\grad u$ (see \cref{ex:potentials} below for a partial argument).

Having \cref{thm:main} at hand, we can now reformulate the nonlinear eigenvalue problem
\begin{align*}
    \lambda_\infty\Phi_{\C_0(\Omega)}\cap\partial J^\infty(u)\neq \emptyset,
\end{align*}
which arises as optimality condition for minimizers of the Rayleigh quotient \labelcref{eq:rayleigh_inf}, in a way that strongly resemble the definition of infinity ground states \labelcref{eq:infty-groundstate}.
For comparison we recap that infinity ground states solve
\begin{align*}
    \min\big(|\nabla u|-\lambda_\infty u,-\Delta_\infty u\big)=0.
\end{align*}

\begin{corollary}
Let $u\in \lip_0(\Omega)$ be a non-negative minimizer of the Rayleigh quotient $R^\infty$, meaning that $R^\infty(u)=\lambda_\infty$. Then it holds
\begin{align}\label{eq:min_equ_inf_minimizers}
    \min\big(\norm{\nabla u}_{\L^\infty}-\lambda_\infty u,-\div(\tau\grad_\tau u)\big)=0
\end{align}
in the sense of measures.
\end{corollary}
\begin{proof}
From \cref{thm:main} we know that $u\geq 0$ solves {$\lambda_\infty\nu u=-\div(\tau\grad_\tau u)$} with a measure ${\nu}\geq 0$ such that ${\nu}(\Omega\setminus\omegamax(u))=0$.
Hence, it holds
\begin{align*}
    -\div{(\tau\grad_\tau u)} = \lambda_\infty{\nu u}
    \begin{cases}
    =0,\quad&\text{in }\Omega\setminus\omegamax(u),\\
    \geq 0,\quad&\text{in }\omegamax(u).
    \end{cases}
\end{align*}
On the other hand, since $u$ minimizes $R^\infty$ we obviously have
\begin{align}
    \norm{\nabla u}_{\L^\infty}-\lambda_\infty u
    \begin{cases}
    \geq 0,\quad&\text{in }\Omega\setminus\omegamax(u),\\
    =0,\quad&\text{in }\omegamax(u).
    \end{cases}
\end{align}
Combining these two equations we obtain \labelcref{eq:min_equ_inf_minimizers}.
\end{proof}
We would like to emphasize that \labelcref{eq:min_equ_inf_minimizers} is strikingly similar to \labelcref{eq:infty-groundstate} for infinity ground states, however, it is valid for {\emph{all non-negative}} minimizers of $R^\infty$. {We will see in \cref{rem:sign-changing-minimizers} below that there may exist global minimizers of $R^\infty$ that change sign.}

\begin{example}[Infinity harmonic potentials]\label{ex:potentials}
Our results can also be applied to infinity harmonic potentials~\labelcref{eq:infinity_harmonic}. 
By definition, these are absolutely minimizing Lipschitz extensions on the open domain $\Omega\setminus\highridge$ and by definition solve
\begin{align}
    \Delta_\infty u = 0\quad\text{on }\Omega\setminus\highridge.
\end{align}
In \cite{evans2005various} (see also \cite{evans1999differential} for a similar result) the authors show that the solution of this equation satisfies the divergence PDE
\begin{align}
    \div(\nu \nabla u) = 0\quad\text{on }\Omega\setminus\highridge,
\end{align}
in a distributional sense.
Here $\nu$ is a non-negative measure concentrated where $\nabla u$ is maximal and it was shown that $\nabla u$ exists on this set.
Applying our \cref{thm:main} shows that
\begin{align}
    \div\left({\tau\nabla_\tau }u\right)=0\quad&\text{on }\Omega\setminus\highridge.
\end{align}
{The following formal argument suggests that one might replace $\grad_\tau u$ by $\grad u$ for an infinity harmonic $u$ so that we recover the result of \cite{evans2005various}. 
The set $\Omegamax(u)$, where $\tau$ is supported,  is a level set of the function $x\mapsto\frac12\abs{\grad u}^2$. 
Assume that $u \in C^2(\Omega)$. Since $u$ is infinity harmonic, it holds that
\begin{align*}
    \langle\grad\frac12\abs{\grad u}^2,\grad u\rangle = \langle D^2 u\grad u,\grad u\rangle = \Delta_\infty u = 0.
\end{align*}
Hence, $\grad u$ is orthogonal to $\grad\frac12\abs{\grad u}^2$ and therefore tangential to the level set $\Omegamax(u)$, which implies $\grad_\tau u = \grad u$.

This computation requires second derivatives of $u$; 
however, infinity harmonic functions are typically not sufficiently smooth, see, e.g., \cite{juutinen1999infinity}. Relating $\grad_\tau u$ of a general infinity harmonic function $u$ to its gradient $\grad u$ without using second derivatives is a challenging topic for future work.}
\end{example}

\begin{remark}[Relation to previous results]
Similar results to \cref{thm:main} can be found in the paper \cite{champion2008infty} and the recent article \cite{katzourakis2021generalised} which appeared during completion of the present work.
In \cite{champion2008infty} the authors investigate infinity ground states \labelcref{eq:infty-groundstate} whereas in \cite{katzourakis2021generalised} the more general problem $\min_{u\in\lip_0(\Omega,\R^N)}{\norm{f(D u)}_{\L^\infty}}/{\norm{g(u)}_{\infty}}$ is studied, which optimizes over vector-valued Lipschitz functions and contains the minimization of \labelcref{eq:rayleigh_inf} as a special case.
Both papers derive similar characterizations of the optimality conditions, relying on the standard and somewhat technical approach of finite $p$ approximation. In contrast, our approach utilizes duality together with simple and elegant arguments from convex analysis. 
Just as we do \cite{champion2008infty} utilizes tangential gradients whereas \cite{katzourakis2021generalised} uses smooth approximation to characterize the gradient on the singular support of the measure $\sigma$.
It is an open question whether this notion of gradient coincides with the tangential gradient from our theory.
\end{remark}

\subsection{Characterization of Subdifferentials}
\label{sec:subdiffs}

For proving \cref{thm:main} we have to characterize the duality map and the subdifferential operator occurring in the eigenvalue problem \labelcref{eq:eigenvalue_prob_inf}. 
To this end, we introduce the set where $u$ attains its maximal modulus:
\begin{definition}\label{def:omegamax}
For $u\in\C_0(\Omega)$ we define
\begin{equation}
    \omegamax(u) =\{x\in\Omega\st |u(x)|=\norm{u}_\infty\}.
\end{equation}
\end{definition}
Because $u$ is continuous, the set $\omegamax(u)$ is closed.
We start with a characterization of the duality map of $\C_0(\Omega)$.

\begin{proposition}[Duality map]\label{prop:duality_map}
Let $\C_0(\Omega)$ be equipped with the norm $\norm{\cdot}_\infty$.
The duality map $\Phi_{\C_0(\Omega)}(u)$ for $u\in\C_0(\Omega)\setminus\{0\}$ consists of all measures $\mu\in\calM(\Omega)$ with $\abs{\mu}(\Omega)=1$ and
\begin{align}
    \label{eq:duality_map_cond_a}
    u\frac{\d\mu}{\d\abs{\mu}} &= \norm{u}_\infty, \quad\abs{\mu}-\text{a.e.},
\end{align}
where $\frac{\d\mu}{\d\abs{\mu}} \in \L^1_{\abs{\mu}}(\Omega)$ is the Radon--Nikod\'ym derivative of $\mu$ w.r.t. its total variation. 
Moreover, any such $\mu$ satisfies
\begin{align}
    \label{eq:duality_map_cond_b}
    \abs{\mu}(\Omega\setminus\omegamax(u))&=0
\end{align}
{and has the polar decomposition
\begin{align}\label{eq:polar_duality_map}
    \mu = \frac{u}{\norm{u}_\infty}\abs{\mu}.
\end{align}
}
\end{proposition}
\begin{proof}
If $\mu\in\calM(\Omega)$ admits $\abs{\mu}(\Omega)= 1$ and \labelcref{eq:duality_map_cond_a} it holds
\begin{align*}
    \langle\mu,u\rangle = \int_\Omega u \d\mu = \int_{\Omega}u\frac{\d\mu}{\d\abs{\mu}}\d\abs{\mu} = \norm{u}_\infty\abs{\mu}(\Omega)=\norm{u}_\infty,
\end{align*}
which implies $\mu\in\Phi_\X(u)$.

Conversely, let us assume that $\mu\in\Phi_\X(u)$.
Then it holds
\begin{align*}
    \norm{u}_\infty = \int_\Omega u \d\mu = \int_\Omega u \frac{\d\mu}{\d\abs{\mu}}\d\abs{\mu} \leq \int_\Omega \abs{u}\d\abs{\mu} \leq \norm{u}_\infty \int_\Omega \d\abs{\mu} \leq \norm{u}_\infty
\end{align*}
and, in particular, all inequalities in this estimate are equalities.
{Since $u\neq 0$}, this implies that the identity \labelcref{eq:duality_map_cond_a} holds true {and that $\abs{\mu}(\Omega)=1$}.
Second, it shows that $\abs{\mu}$-\text{a.e.} it holds $\abs{u}=\norm{u}_\infty$ which is equivalent to~\labelcref{eq:duality_map_cond_b}.
{%
We finish the proof by computing the polar decomposition \labelcref{eq:polar_duality_map}.
For this we compute
\begin{align*}
    \norm{u}_\infty = u\frac{\d\mu}{\d\abs{\mu}}\leq\abs{u}\leq\norm{u}_\infty,\quad \abs{\mu}-a.e.
\end{align*}
Hence, equality holds true and $u$ and $\frac{\d\mu}{\d\abs{\mu}}$ have the same sign $\abs{\mu}$-a.e. This implies
\begin{align*}
    \mu=\frac{\d\mu}{\d\abs{\mu}}\abs{\mu}=\frac{u}{\norm{u}_\infty}\abs{\mu}.
\end{align*}
}
\end{proof}
Now we characterize the subdifferential of $J^\infty(u)$, which is significantly more involved. 
We first prove an integral characterization similar to \cite{bredies2016pointwise}---which deals with the subdifferential of the total variation functional---and then prove a pointwise one.
The main insight from the following integral characterization is that the space of divergence-measure fields, defined in \cref{def:divergence-measure}, is strongly connected to the subdifferential of $J^\infty$.

\begin{proposition}[Integral characterization of $\partial J^\infty$]\label{prop:subdifferential}
Let $u\in \lip_0(\Omega)$. Then it holds
\begin{align}\label{eq:subdifferential_J_inf}
    \partial J^\infty(u)=\left\lbrace-\div\sigma\st\sigma\in\calD\calM(\Omega,\R^n),\;\langle-\div\sigma,u\rangle=J^\infty(u),\;|\sigma|(\Omega)\leq 1\right\rbrace.
\end{align}
A measure $\sigma\in\calD\calM(\Omega,\R^n)$ such that $-\div\sigma\in\partial J^\infty(u)$ is called \emph{calibration} of $u$.
\end{proposition}
\begin{proof}
Following \cite{bredies2016pointwise} we have to show that the closure of the set
\begin{align}
    C\defeq \left\lbrace -\div\sigma\st\sigma\in \C^\infty(\closure{\Omega}),\;\norm{\sigma}_{\L^1}\leq 1 \right\rbrace
\end{align}
with respect to the total variation norm on $\calM(\Omega)$ is given by
\begin{align}
    \closure{C}=\{-\div \sigma\st \sigma\in \calD\calM(\Omega,\R^n),\;\abs{\sigma}(\Omega)\leq 1\}=:K.
\end{align}
Since $C\subset K$ is obviously true, we first show that $K$ is closed which implies $\closure{C}\subset K$.
To this end, let us take a sequence $\{\sigma_n\}_{n\in\N}\subset \calM(\Omega,\R^n)$ such that $\abs{\sigma_n}(\Omega)\leq 1$ and $-\div\sigma_n\to\mu\in\calM(\Omega)$. 
Then there is a vector-valued Radon measure $\sigma\in\calM(\Omega,\R^n)$ such that (up to a subsequence that we do not relabel) $\sigma_n\wsto\sigma$ and hence $|\sigma|(\Omega)\leq 1$, by lower semicontinuity of the total variation. 
Moreover, we obtain
\begin{align*}
    \langle \sigma,\nabla\varphi\rangle=\lim_{n\to\infty}\langle \sigma_n,\nabla\varphi\rangle=\lim_{n\to\infty}\langle-\div \sigma_n,\varphi\rangle=\langle\mu,\varphi\rangle,\quad\forall\varphi\in \C^\infty_c(\Omega),
\end{align*}
which means that $\mu=-\div\sigma$. 
To show that $K\subset\closure{C}$ it suffices to prove that 
\begin{align*}
    \langle-\div\sigma, u\rangle\leq J^\infty(u),\quad\forall u\in\dom(J^\infty).
\end{align*}
By~\cite{bungert2020structural} we can find a sequence $\{u_n\}_{n\in\N}\subset \C^\infty_c(\Omega)$ which satisfies $\norm{u_n-u}_\infty\to 0$ as $n\to\infty$ and $J^\infty(u_n)\leq J^\infty(u)$. Then it holds
\begin{align*}
    \langle-\div\sigma,u\rangle=\lim_{n\to\infty}\langle-\div\sigma,u_n\rangle=\lim_{n\to\infty}\langle\sigma,\nabla u_n\rangle\leq \lim_{n\to\infty} J^\infty(u_n)\leq J^\infty(u),
\end{align*}
which lets us conclude.
\end{proof}

\sloppy 
Before we proceed with a pointwise characterization of calibrations $\sigma$ which satisfy $-\div\sigma\in\partial J^\infty(u)$, we need to understand how the integration-by-parts formula
\begin{align*}
    \int_\Omega \nabla u \cdot\d\sigma = -\int_\Omega u\d\div\sigma = J^\infty(u)
\end{align*}
can be made rigorous.
Assuming for a moment that the formula is valid one can show analogously to \cref{prop:duality_map} that $\sigma$ is parallel to $\nabla u$ and that $\sigma$ is supported where $\abs{\nabla u}$ is maximal.

The problem with this formula is that integral of the $\L^\infty$-function $\nabla u$ with respect to the (non absolutely continuous) measure $\sigma$ is not well-defined.
This can be fixed by replacing $\nabla u$ with the \emph{tangential gradient} with respect to $\abs{\sigma}$, {a concept that goes back to~\cite{bouchitte1997energies}}.
We will use the following definition, which is a slight modification of~\cite[Def. 4.6]{jimenez:2008}. For details, we refer the reader to~\cite{jimenez:2008, champion2008infty}. 
{We would also like to point to \cite{luvcic2021characterisation} for a novel characterization of the tangential gradient as minimal norm element of a set-valued gradient operator.}
\begin{definition}[Tangential gradient]\label{def:tangetial-grad}
Let $u \in \lip_0(\Omega)$ and $\mu \in \calM_+(\Omega)$ be a non-negative measure. 
Let  $\{u_n\}_{n\in\N} \subset \C_c^\infty(\Omega)$ be any sequence such that
\begin{subequations}\label{eq:conv-def-tangential-grad}
\begin{align}
    &\sup_{n \in \N} \max_{x \in \Omega}\abs{\grad u_n(x)} \leq J^\infty(u), \label{eq:conv-def-tangential-grad-a} \\
    & \max_{x \in \Omega} \abs{u_n(x) - u(x)} \to 0.
\end{align}
\end{subequations}
Denote by $P_\mu(x, \nabla u_n(x))$ the projection of the gradient $\nabla u_n(x)$ onto the tangent space of $\mu$ at $x \in \Omega$ (see~\cite[Def. 4.3]{jimenez:2008} {and \cite{fragala1999some} for relations to other notions of tangent spaces to a measure}). The \emph{tangential gradient} of $u$ with respect to $\mu$ is defined as the following limit with respect to the weak$^*$ convergence in $\L^\infty_\mu(\Omega,\R^n)$
\begin{equation*}
    \nabla_\mu u \defeq  \text{weak}^*\text{-}\lim_{n \to \infty} P_\mu(\cdot, \nabla u_n(\cdot)).
\end{equation*}
By~\cite[Prop. 4.5]{jimenez:2008}, this limit exists and does not depend on the choice of the approximating sequence $\{u_n\}_{n\in\N}$. We note that the operation $\phi(\cdot) \mapsto P_\mu(\cdot, \phi(\cdot))$ is nonlinear in $\phi$.
\end{definition}
\begin{lemma}\label{lem:norm-tangential-grad}
Let $u \in \lip_0(\Omega)$ and $\mu \in \calM_+(\Omega)$. Then
\begin{equation*}
    \abs{\nabla_\mu u} \leq J^\infty(u) \quad \text{$\mu$-\text{a.e.}}
\end{equation*}
\end{lemma}
\begin{proof}
Since $P_\mu(\cdot, \nabla u_n(\cdot)) \wsto \nabla_\mu u$ and norms are weakly* lower-semicontinuous, we get
\begin{equation*}
    \norm{\nabla_\mu u}_{\L^\infty_\mu} \leq \liminf_{n \to \infty} \norm{P_\mu(\cdot, \nabla u_n(\cdot))}_{\L^\infty_\mu} \leq \liminf_{n \to \infty} \norm{{\nabla u_n}}_{\L^\infty_\mu} \leq J^\infty(u),
\end{equation*}
which implies the claim. Here we used the fact that $\abs{P_\mu(x, \nabla u_n(x))} \leq \abs{\nabla u_n(x)}\leq J^\infty(u)$ for all $x\in\Omega$, using condition~\labelcref{eq:conv-def-tangential-grad-a}.
\end{proof}

\begin{proposition}[Integration by parts]\label{prop:int_by_parts}
Let $u \in \lip_0(\Omega)$ and $\sigma \in \calD\calM(\Omega,\R^n)$. Then
\begin{align}
    -\int_\Omega u\d\div\sigma = \int_\Omega \nabla_{\abs{\sigma}} u \cdot\d{\sigma}.
\end{align}
\end{proposition}
\begin{proof}
The proof is a straightforward adaption of \cite[Prop. 4.10]{jimenez:2008}.
Let $\{u_n\}_{n\in\N}\subset\C_c^\infty(\Omega)$ be a sequence satisfying \labelcref{eq:conv-def-tangential-grad}.
Then we can compute
\begin{align*}
    -\int_\Omega u\d\div\sigma 
    &= -\lim_{n\to\infty}\int_\Omega u_n\d\div\sigma = \lim_{n\to\infty}\int_\Omega \nabla u_n \cdot\d\sigma\\
    &= \lim_{n\to\infty}\int_\Omega \nabla u_n \cdot\frac{\d\sigma}{\d\abs{\sigma}}\d\abs{\sigma}.
\end{align*}
Since $\abs{\sigma}$-a.e. the function $\frac{\d\sigma}{\d\abs{\sigma}}$ lies in the tangent space of $\abs{sigma}$ (see \cite[Lem.~4.9]{jimenez:2008}), we get 
\begin{align*}
    \nabla u_n \cdot \frac{\d\sigma}{\d\abs{\sigma}} = P_{\abs{\sigma}}(\cdot,\nabla u_n(\cdot))\cdot \frac{\d\sigma}{\d\abs{\sigma}},\quad\abs{\sigma}-\text{a.e.}
\end{align*}
and hence by the definition of the tangential gradient:
\begin{align*}
   -\int_\Omega u\d\div\sigma 
   &=  \lim_{n\to\infty}\int_\Omega P_{\abs{\sigma}}(\cdot,\nabla u_n(\cdot)) \cdot\frac{\d\sigma}{\d\abs{\sigma}}\d\abs{\sigma} \\
   &= \int_\Omega \nabla_{\abs{\sigma}} u\cdot\frac{\d\sigma}{\d\abs{\sigma}}\d\abs{\sigma} = \int_\Omega \nabla_{\abs{\sigma}} u\cdot\d{\sigma}.
\end{align*}
\end{proof}
With the same approximation trick, we can define the set where a function $u \in \lip_0(\Omega)$ attains the maximal value of its gradient in a sense that will become clear in \cref{prop:subdiff-pointwise}. 

\begin{definition}\label{def:Omegamax}
Let $u \in \lip_0({\Omega})$ and consider any sequence $\{u_n\}_{n\in\N} \subset \C_c^\infty(\Omega)$ satisfying~\labelcref{eq:conv-def-tangential-grad}. 
We define
\begin{equation*}
    \Omegamax(u) \defeq   \{x \in \Omega \colon \limsup_{n \to \infty} \abs{\grad u_n(x)} = \norm{\nabla u}_{\L^\infty}\}.
\end{equation*}
Similarly to \cref{def:tangetial-grad}, this definition does not depend on the choice of the approximating sequence (cf.~\cite[Prop. 4.5]{jimenez:2008}).
\end{definition}
\begin{remark}
This definition bears similarities with the attainment set defined in \cite{brizzi2021property} as
\begin{align*}
    \mathcal{A}(u) \defeq \{x\in\Omega \st \abs{\nabla u}(x)=\norm{\nabla u}_{\L^\infty}\}.
\end{align*}
Here $x\mapsto\abs{\nabla u}(x)$ denotes an \emph{everywhere-defined} version of the $\L^\infty$-function $x\mapsto\abs{\nabla u(x)}$, defined as
\begin{align*}
    \abs{\nabla u}(x)\defeq\lim_{r\downarrow 0}\inf\left\lbrace\lambda>0\st u(y)-u(x)\leq \lambda\abs{y-x},\;\forall y\in B_r(x)\right\rbrace,\quad x\in\Omega.
\end{align*}
It is not unlikely that under suitable regularity conditions the sets $\Omegamax(u)$ and $\mathcal{A}(u)$ coincide, however, for dealing with tangential gradients our definition is more useful.
\end{remark}

We continue with a few examples that illustrate the definition of $\Omegamax(u)$.

\begin{figure*}[t!]%
\centering
\begin{subfigure}[t]{0.23\textwidth}
\centering
    \begin{tikzpicture}%
    \begin{axis}[scale=.5, nice axis, ticks=none, ymin=-.1, ymax=1.1, xmin=-1.2, xmax=1.2]%
            \addplot[domain=-1:1, line] {1-abs(x)};%
    \end{axis}%
    \end{tikzpicture}
    \caption{$\Omegamax = \Omega \setminus \{0\}$}
    \label{fig:dist-fun-1d}
\end{subfigure}
\begin{subfigure}[t]{0.23\textwidth}
\centering
    \begin{tikzpicture}%
    \begin{axis}[scale=.5, nice axis, ticks=none, ymin=-.1, ymax=1.1, xmin=-1.2, xmax=1.2]%
            \addplot[domain=-1:1, line] {1-2*abs(x)+x^2};%
    \end{axis}%
    \end{tikzpicture}
    \caption{$\Omegamax = \emptyset$}
    \label{fig:peak-1d}
\end{subfigure}
\begin{subfigure}[t]{0.25\textwidth}
\centering
    \begin{tikzpicture}
    \begin{axis}[scale=.7, anchor=origin, zmin=-.1, zmax=1.1, xmin=-1.2, xmax=1.2, ymin=-1.2, ymax=1.2, axis lines = none, clip=false]
          \begin{scope}[canvas is xy plane at z=0,>=stealth]
         \draw (-1,-1) -- (1,-1) -- (1,1) -- (-1,1) -- (-1,-1);
         \draw[dashed] (-1,-1) -- (1,1);
         \draw[dashed] (1,-1) -- (-1,1);
        \end{scope} 

    \draw[dashed] (0,0,0) -- (0,0,1);
    \addplot3[surf, colormap = {whiteblack}{color(0cm)  = (white);color(1cm) = (black)}, opacity=0.5, domain=-1:1, y domain=-1:1] ({x}, {y}, { 1 - max(abs(x),abs(y)) });
    \draw[thick] (-1,-1,0) -- (1,-1,0) -- (1,1,0);
    \draw[thick, dashed] (-1,-1,0) -- (-1,1,0) -- (1,1,0);
    \draw[thick] (-1,-1,0) -- (0,0,1) -- (1,1,0);
    \draw[thick] (1,-1,0) -- (0,0,1);
    \draw[thick,dashed] (-1,1,0) -- (0,0,1);
    \end{axis}   
    \end{tikzpicture}
    \caption{$\Omegamax = \Omega \setminus \{\text{diags}\}$}
    \label{fig:dist-fun-square}
\end{subfigure}
\begin{subfigure}[t]{0.24\textwidth}
\centering
    \begin{tikzpicture}
    \begin{axis}[scale=.7, anchor=origin, zmin=-.1, zmax=1.1, xmin=-1.2, xmax=1.2, ymin=-1.2, ymax=1.2, axis lines = none, clip=false]
          \begin{scope}[canvas is xy plane at z=0,>=stealth]
         \draw (-1,-1) -- (1,-1) -- (1,1) -- (-1,1) -- (-1,-1);
        \end{scope} 
    \addplot3[surf, colormap = {whiteblack}{color(0cm)  = (white);color(1cm) = (black)}, opacity=0.5, domain=-1:1, y domain=-1:1] ({x}, {y}, {(1-2*abs(x)+x^2)*min(2 - 2*abs(y),1) });
    
    \draw[thick] (-1,1,0) -- (-1,-1,0) -- (1,-1,0) -- (1,1,0);
    \draw[thick, dashed] (-1,1,0) -- (1,1,0);
    \draw[thick] (0,-1,0) -- (0,-.5,1) -- (0,.5,1);
    \draw[thick, dashed] (0,.5,1) -- (0,1,0);
    \addplot3[domain=-1:1, samples y=0] ({x},{-0.5},{1-2*abs(x)+x^2});
    \addplot3[domain=0:1, samples y=0] ({x},{0.5},{1-2*abs(x)+x^2});
    \addplot3[domain=-1:0, samples y=0, dashed] ({x},{0.5},{1-2*abs(x)+x^2});
    \addplot3[domain=0:1, samples y=0, thick] ({x-1},{x/2-1},{x^3});
    \addplot3[domain=0:1, samples y=0, thick] ({1-x},{x/2-1},{x^3});
    \addplot3[domain=0:1, samples y=0, dashed, thick] ({x-1},{1-x/2},{x^3});
    \addplot3[domain=0:1, samples y=0, thick] ({1-x},{1-x/2},{x^3});
    
    \node[anchor=north east,scale=0.8] at (0,-1,0) {$x$};
    \node[anchor=north west,scale=0.8] at (1,0,0) {$y$};
    
    \node[anchor=south west,scale=0.8] at (0,-1,0) {$A$};
    \node[anchor=south east,scale=0.8] at (0,-0.5,1) {$A^\prime$};
    \node[anchor=north,scale=0.8] at (0,1,0) {$B$};
    \node[anchor=south east,scale=0.8] at (0,0.5,1) {$B^\prime$};
    \end{axis}   
    \end{tikzpicture}
    \caption{$\Omegamax = P_{xy}(\overline{AA^\prime}) \cup  P_{xy}(\overline{BB^\prime})$. $P_{xy}$ is projection onto $xy$-plane.}
    \label{fig:combined-2d}
\end{subfigure}
\caption{Functions from \cref{ex:dist-fun-1d,ex:peak-1d,ex:dist-fun-square,ex:combined-2d}}
\end{figure*}

\begin{example}[Distance function of an interval]\label{ex:dist-fun-1d}
Let $\Omega \defeq (-1,1)$ and $u(x) \defeq \d_\Omega(x) = 1 - \abs{x}$ be the distance function (sketched in \cref{fig:dist-fun-1d}). Choosing $u_n$ to be a sequence of smooth symmetric approximations satisfying~\labelcref{eq:conv-def-tangential-grad}, we see that $\lim_{n \to \infty} \abs{u'_n(x)} = 1$ for $x \neq 0$ and $u'_n(0) = 0$ for all $n$. Therefore, $\Omegamax(u) = (-1,0) \cup (0,1)$.
\end{example}

\begin{example}[A function with empty $\Omegamax$]\label{ex:peak-1d}
Let $\Omega \defeq (-1,1)$ and $u(x)=1 - 2\abs{x}+ x^2$ (sketched in \cref{fig:peak-1d}). 
In this case $\abs{u'}$ increases towards the origin and the maximal value is attained at zero, where $u$ is not differentiable. Choosing again $u_n$ to be a sequence of smooth symmetric approximations satisfying~\labelcref{eq:conv-def-tangential-grad}, we see that 
\begin{equation*}
\lim_{n \to \infty} \abs{u'_n(x)} = 
\begin{cases}
2\abs{x - \sign(x)} < 2 = J^\infty(u), \quad & x \in (-1,0)\cup(0,1),\\
0, \quad & x = 0.
\end{cases}
\end{equation*}
Hence $\Omegamax(u)=\emptyset$.
\end{example}

\begin{example}[Distance function of a square]\label{ex:dist-fun-square}
Let $\Omega \defeq (-1,1)^2$ and $u(x) \defeq \d_\Omega(x) = 1 - \max(\abs{x},\abs{y})$. A sketch is shown in \cref{fig:dist-fun-square}. Choose again a (radially) symmetric smooth approximating sequence $u_n$. Since $u$ is differentiable everywhere except for the diagonals $\overline{(-1,-1)(1,1)}$ and $\overline{(-1,1)(1,-1)}$, we have $\lim_{n\to\infty} \grad u_n(x) = \grad u(x)$ everywhere except for the diagonals. A calculation shows that on the diagonals $\lim_{n\to\infty} \abs{\grad u_n(x)} = \frac{1}{\sqrt{2}}$, i.e.,
\begin{equation*}
 \lim_{n \to \infty} \abs{\grad u_n(x)} = 
\begin{cases}
1 = J^\infty(u), \quad & x \notin \text{diagonals},\\
\frac{1}{\sqrt{2}} < J^\infty(u), \quad & x \in \text{diagonals}.
\end{cases}
\end{equation*} 
Hence, $\Omegamax(u) = \Omega \setminus \{\text{diagonals}\}$.
\end{example}

\begin{example}(Mountain ridge)\label{ex:combined-2d}
Let $\Omega = (-1,1)^2$ and denote 
\begin{equation*}
    \varphi(x) \defeq 1 - 2\abs{x}+ x^2, \quad \psi(y) \defeq 
    \begin{cases}
    2(y+1), \quad & y \in (-1,-0.5),\\
    1, \quad & y \in (-0.5,0.5),\\
    2(1-y), \quad & y \in (0.5,1).
    \end{cases}
\end{equation*}
Let $u(x,y) \defeq \varphi(x) \psi(y)$ (sketched in \cref{fig:combined-2d}). From \cref{ex:peak-1d} we know that the partial derivative $\frac{\partial u}{\partial x}$ does not exist at $x=0$ and that for a symmetric approximating sequence $u_n$ we have
\begin{equation*}
    \lim_{n\to\infty} \left. \frac{\partial u_n}{\partial x} \right|_{x=0} = 0 < J^\infty(u) = 2.
\end{equation*}
Since $u(0,y)=1$ for $y \in \left(-\frac12,\frac12\right)$, we have $\left. \frac{\partial u}{\partial y}\right|_{x=0} = 0$ and $\lim_{n\to\infty} \left. \frac{\partial u_n}{\partial y} \right|_{x=0} = 0$ for $y \in \left(-\frac12,\frac12\right)$. Hence,
\begin{equation*}
    \lim_{n\to\infty} \abs{\grad u_n(x,y)} = 0 < J^\infty(u) = 2 \quad \text{for $x=0$, $y \in \left(-\frac12,\frac12\right)$.}
\end{equation*}
For $y \in \left(-1,-\frac12\right)\cup\left(\frac12,1\right)$, however, we have $\left| \frac{\partial u}{\partial y}\right|(0,y) = 2$ and it is easy to convince oneself that
\begin{equation*}
    \lim_{n\to\infty} \abs{\grad u_n(x,y)} = 2 = J^\infty(u) \quad \text{for $x=0$, $y \in \left(-1,-\frac12\right)\cup\left(\frac12,1\right)$.}
\end{equation*}
Therefore, we conclude that
\begin{equation*}
    \Omegamax(u) = \left \{(x,y) \colon x=0, \, y \in \left(-1,-\frac12\right)\cup\left(\frac12,1\right) \right\}.
\end{equation*}
In other words, $\Omegamax(u)$ consists of projections onto the $xy$-plane of the two open segments $\overline{AA^\prime}$ and $\overline{BB^\prime}$ shown in \cref{fig:combined-2d}.
We note that on the whole set $\Omegamax(u)$ the gradient $\nabla u$ \emph{does not exist}. 
In contrast, according to \cite{evans2005various} the gradients of an infinity harmonic potential exist on the so-called contact set $P_{xy}(\overline{AA'})\cup P_{xy}(\overline{BB'})$ which implies that the function $u$ which we constructed is no infinity harmonic potential.
\end{example}

\begin{example}[Distance function of fat Cantor set]
Let $F\subset[0,1]$ be a Cantor set \cite{smith1874integration} and let $u(x)=\dist(x,F)$ be the distance function to $F$.
In this case $\Omegamax(u)=[0,1]\setminus(F\cup D)$ where $D$ is a countable discrete set, corresponding to the midpoints and boundary points of the intervals which are removed from $[0,1]$ to construct the Cantor set. 
We note that $\Omegamax(u)$ is dense in $[0,1]$.
\end{example}

Now we are ready to give a pointwise characterisation of the subdifferential of $J^\infty$. 

\begin{proposition}[Pointwise characterization of calibrations, Part 1]\label{prop:subdiff-pointwise}
Let $u\in \lip_0(\Omega)\setminus\{0\}$ and $\sigma\in\calD\calM(\Omega,\R^n)$.
It holds $-\div\sigma \in \partial J^\infty(u)$ if and only if $\abs{\sigma}(\Omega)=1$ and
\begin{equation}\label{eq:pw_subdiff-0}
     \nabla_{\abs{\sigma}} u \cdot \frac{\d\sigma}{\d\abs{\sigma}} = J^\infty(u) \quad  \abs{\sigma}-\text{a.e.}, 
\end{equation}
where $\nabla_{\abs{\sigma}}$ denotes the tangential gradient w.r.t. $\abs{\sigma}$ and $\frac{\d\sigma}{\d\abs{\sigma}} \in \L^1_{\abs{\sigma}}(\Omega)$ is the Radon--Nikod\'ym derivative of $\sigma$ w.r.t. its total variation. 
Moreover, any such $\sigma$ satisfies
\begin{equation}\label{eq:sigma_support}
    \abs{\sigma}(\Omega \setminus \Omegamax(u)) = 0
\end{equation}
{and has the polar decomposition
\begin{align}\label{eq:polar_subdiff}
    \sigma = \frac{\grad_{\abs{\sigma}}u}{\norm{\grad u}_{\L^\infty}}\abs{\sigma}.
\end{align}}
\end{proposition}
\begin{proof}
If $\sigma$ fulfills~\labelcref{eq:pw_subdiff-0} and $\abs{\sigma}(\Omega)=1$, we obtain, using \cref{prop:int_by_parts}, that 
\begin{equation*}
    -\int_\Omega u\d\div\sigma = \int_{\Omega} \nabla_{\abs{\sigma}} u \cdot \frac{\d\sigma}{\d\abs{\sigma}} \d\abs{\sigma} = \int_{\Omega} J^\infty(u) \d\abs{\sigma} = J^\infty(u).
\end{equation*}
By \cref{prop:subdifferential}, this implies $-\div\sigma\in\partial J^\infty(u)$.

Conversely, suppose that $-\div\sigma\in\partial J^\infty(u)$. 
{By \cref{prop:subdifferential} we know $\abs{\sigma}(\Omega)\leq 1$.}
By mollification we can obtain a sequence
$\{u_n\}_{n\in\N} \subset \C_c^\infty(\Omega)$ that satisfies~\labelcref{eq:conv-def-tangential-grad} as in~\cite{bungert2020structural}. 
By the definition of the tangential gradient, for any $\psi \in \L^1_{\abs{\sigma}}(\Omega)$ it holds
\begin{equation*}
    \int_\Omega \psi(x) \cdot \nabla_{\abs{\sigma}} u(x) \d\abs{\sigma}(x) = \lim_{n \to \infty} \int_\Omega \psi(x) \cdot P_{\abs{\sigma}}(x, \nabla u_n(x)) \d\abs{\sigma}(x).
\end{equation*}
Utilizing this for $\psi=\frac{\d\sigma}{\d\abs{\sigma}}\in \L^1_{\abs{\sigma}}(\Omega)$ we get by \cref{prop:int_by_parts}
\begin{align}\label{eq:est-subdiff-ness}
    \begin{split}
    J^\infty(u) &= -\int_\Omega u\d(\div\sigma) \\
    &= \int_{\Omega} \nabla_{\abs{\sigma}} u\cdot \frac{\d\sigma}{\d\abs{\sigma}} \d\abs{\sigma} \\
    &= \lim_{n \to \infty} \int_{\Omega} P_{\abs{\sigma}}(x, \nabla u_n(x)) \cdot \frac{\d\sigma}{\d\abs{\sigma}}(x) \d\abs{\sigma}(x) \\
    &\leq \limsup_{n \to \infty} \int_{\Omega} \abs{P_{\abs{\sigma}}(x, \nabla u_n(x))} \d\abs{\sigma}(x) \\
    &\leq \limsup_{n \to \infty} \int_{\Omega} \abs{\nabla u_n(x)} \d\abs{\sigma}(x) \\
    &\leq \int_{\Omega} \limsup_{n \to \infty}  \abs{\nabla u_n(x)}\d\abs{\sigma}(x) \\
    &\leq \int_{\Omega} J^\infty(u) \d\abs{\sigma} \\
    &= J^\infty(u).
    \end{split}
\end{align}
Therefore, all inequalities are satisfied as equalities and in particular, {since $u\neq 0$, it holds $\abs{\sigma}(\Omega)=1$. Furthermore,}
\begin{equation}\label{eq:parallelity-integral}
    \int_{\Omega} \nabla_{\abs{\sigma}} u\cdot \frac{\d\sigma}{\d\abs{\sigma}} \, \d\abs{\sigma} = \int_{\Omega} J^\infty(u) \, \d\abs{\sigma}.
\end{equation}
By \cref{lem:norm-tangential-grad} we have that
\begin{equation*}
    \nabla_{\abs{\sigma}} u\cdot \frac{\d\sigma}{\d\abs{\sigma}} \leq \abs{\nabla_{\abs{\sigma}} u} \leq J^\infty(u) \quad \text{$\abs{\sigma}$-\text{a.e.},}
\end{equation*} 
hence, {since $u\neq 0$}, equality in~\labelcref{eq:parallelity-integral} is only possible if~\labelcref{eq:pw_subdiff-0} holds.
Furthermore,~\labelcref{eq:est-subdiff-ness} implies that
\begin{equation*}
    \limsup_{n \to \infty}  \abs{\nabla u_n(x)} = J^\infty(u) \quad \abs{\sigma}-\text{a.e.},
\end{equation*}
and therefore~\labelcref{eq:sigma_support} holds.

{For the polar decomposition \labelcref{eq:polar_subdiff} we compute using \cref{lem:norm-tangential-grad}:
\begin{align*}
    \norm{\nabla u}_{\L^\infty}=\nabla_{\abs{\sigma}}u\cdot\frac{\d\sigma}{\d\abs{\sigma}} \leq \abs{\nabla_{\abs{\sigma}}u}\leq \norm{\nabla u}_{\L^\infty},\quad\abs{\sigma}-a.e.
\end{align*}
Hence, equality holds true and $\nabla_{\abs{\sigma}}u$ and $\frac{\d\sigma}{\d\abs{\sigma}}$ are linearly dependent $\abs{\sigma}$-a.e. which implies
\begin{align*}
    \sigma=\frac{\d\sigma}{\d\abs{\sigma}}\abs{\sigma} = \frac{\nabla_{\abs{\sigma}}u}{\norm{\nabla u}_{\L^\infty}}\abs{\sigma}.
\end{align*}}
\end{proof}

In the previous proposition we have proved that calibrations $\sigma$ are concentrated in the set $\Omegamax(u)$ and parallel to the tangential gradient with respect to $\abs{\sigma}$. {We started by defining the conditional gradient  \cref{def:tangetial-grad} using an approximating sequence and then obtained an integration by parts formula in \cref{prop:int_by_parts}, which we later used in \cref{prop:subdiff-pointwise}. An alternative and in some sense complimentary route is to use the pairing measure $\llangle\nabla u,\sigma\rrangle$, which introduced by \v{S}ilhav{\'y} in \cite{vsilhavy2008normal} and for which the integration by parts formula is part of the definition.}

\begin{proposition}[\v{S}ilhav{\'y} \cite{vsilhavy2008normal,vsilhavy2009divergence}]\label{prop:pairing}
Let $u\in\lip_0({\Omega})$ and $\sigma\in\calD\calM(\Omega,\R^n)$.
There exists a unique signed measure $\llangle\nabla u,\sigma\rrangle\in\calM(\Omega)$ such that
\begin{align}
    -\int_\Omega u\d\div\sigma &= \int_\Omega\d\llangle\nabla u,\sigma\rrangle, \\
    \abs{\llangle\nabla u,\sigma\rrangle}(U) &\leq \norm{\nabla u}_{\L^\infty(U)}\abs{\sigma}(U),\qquad\forall U\subset\Omega\text{ open}.
\end{align}
\end{proposition}
\begin{proof}
The first statement is a special case of \cite[Thm. 2.3]{vsilhavy2009divergence}.
The second one can be found in \cite[Prop. 5.2]{vsilhavy2008normal}.
\end{proof}

{We can also formulate \cref{prop:subdiff-pointwise}  in the language of the above pairing and obtain a second pointwise characterization of calibrations $\sigma$.}

\begin{proposition}[Pointwise characterization of calibrations, Part 2]\label{prop:subdiff-pointwise-pair}
Let $u\in \lip_0({\Omega})\setminus\{0\}$ and $\sigma\in\calD\calM(\Omega,\R^n)$. 
It holds $-\div\sigma \in \partial J^\infty(u)$ if and only if $\abs{\sigma}(\Omega)=1$ and
\begin{align}
        \llangle\nabla u, \sigma\rrangle &= J^\infty(u) |\sigma|.\label{eq:pw_subdiff_pair}
\end{align}
Here $\llangle\nabla u , \sigma\rrangle\in\calM(\Omega)$ denotes the pairing from \cref{prop:pairing}.
\end{proposition}
\begin{proof}
If $\sigma\in\calD\calM(\Omega,\R^n)$ fulfills~\labelcref{eq:pw_subdiff_pair} and $\abs{\sigma}(\Omega)=1$, we obtain from \cite[Thm. 2.3]{vsilhavy2009divergence}
\begin{equation*}
    -\int_\Omega u\d\div\sigma = \int_{\Omega} \d\llangle\nabla u,\sigma\rrangle = J^\infty(u) \int_{\Omega}  \d\abs{\sigma} = J^\infty(u).
\end{equation*}
By \cref{prop:subdifferential}, this implies $-\div\sigma\in\partial J^\infty(u)$.

Let us now assume that $-\div\sigma\in\partial J^\infty(u)$ for $\sigma\in\calD\calM(\Omega,\R^n)$.
{By \cref{prop:subdiff-pointwise} we know that $\abs{\sigma}(\Omega)=1$.}
According to \cref{prop:pairing} it holds
\begin{align*}
  \abs{\llangle\nabla u , \sigma\rrangle}(U)\leq J^\infty(u)\abs{\sigma}(U),\quad\forall U\subset\Omega\text{ open,}
\end{align*}
which immediately implies
\begin{align*}
    \llangle\nabla u,\sigma\rrangle(U)\leq J^\infty(u)\abs{\sigma}(U),\quad\forall U\subset\Omega\text{ open}.
\end{align*}
Using outer regularity of the measures \cite{tao2010epsilon}, this implies that we have the following inequality on all Borel sets
\begin{align*}
     \llangle\nabla u,\sigma\rrangle \leq J^\infty(u)\abs{\sigma}.
\end{align*}
To show equality, let us assume that there exists $\eps>0$ {and a Borel set $B\subset\Omega$ with $\abs{\sigma}(B)> 0$ such that $\llangle\nabla u,\sigma\rrangle(B)\leq(1-\eps)J^\infty(u)\abs{\sigma}(B)$.
Then it follows
\begin{align*}
    J^\infty(u) 
    &= 
    -\int_\Omega u\d\div\sigma 
    = 
    \int_{\Omega} \d\llangle\nabla u,\sigma\rrangle \leq
    J^\infty(u) \abs{\sigma}(\Omega\setminus B)+
    (1-\eps)\,J^\infty(u)\abs{\sigma}(B) \\
    &\leq (1-\eps\abs{\sigma}(B))\,J^\infty(u).
\end{align*}
Since $u\neq 0$,} this is a contradiction and hence we have shown \labelcref{eq:pw_subdiff_pair}.
\end{proof}

{%
We have made all the necessary preparations to prove out main result, \cref{thm:main}.
\begin{proof}[Proof of \cref{thm:main}]
According to \cref{prop:duality_map,prop:subdifferential} a function $u\in\C_0(\Omega)\setminus\{0\}$ solves $\lambda\Phi_{\C_0(\Omega)}(u)\cap\partial J^\infty(u)\neq \emptyset$ if and only if there exist measures $\mu\in\M(\Omega)$ and $\sigma\in\calD\calM(\Omega,\R^n)$ with $\lambda\mu = -\div\sigma$.
Furthermore, in \cref{prop:duality_map} it was proved that $\mu = \frac{u}{\norm{u}_\infty}\abs{\mu}$.
Therefore, the measure $\nu := \frac{1}{\norm{u}_\infty}\abs{\mu}$ satisfies $\nu(\Omega)=\frac{1}{\norm{u}_\infty}$ and also 
\begin{align*}
    \nu(\Omega\setminus\omegamax(u)) = \frac{1}{\norm{u}_\infty}\abs{\mu}(\Omega\setminus\omegamax(u)) = 0.
\end{align*}
Analogously, \cref{prop:subdiff-pointwise} implies that $\sigma =  \frac{\nabla_{\abs{\sigma}}u}{\norm{\grad u}_{\L^\infty}}\abs{\sigma}$.
Hence, the measure $\tau := \frac{1}{\norm{\grad u}_{\L^\infty}}\abs{\sigma}$ satisfies $\tau(\Omega) = \frac{1}{\norm{\grad u}_{\L^\infty}}$ and also
\begin{align*}
    \tau(\Omega\setminus\Omegamax(u)) = \frac{1}{\norm{\grad u}_{\L^\infty}}\abs{\sigma}(\Omega\setminus\Omegamax(u)) = 0.
\end{align*}
Furthermore, since $\abs{\sigma}$ and $\tau$ differ just by a non-zero constant multiple, it also holds that $\grad_{\abs{\sigma}}u=\grad_{\tau}u$.
Hence, $\lambda\mu = -\div\sigma$ is equivalent to $\lambda\nu u = -\div(\tau\grad_\tau u)$ with the above choices of $\nu$ and $\tau$.

Finally, the statement that $\sigma = \tau\grad_\tau u$ satisfies $\llangle\grad u,\sigma\rrangle = \norm{\grad u}_{\L^\infty}\abs{\sigma}$ is the statement of \cref{prop:subdiff-pointwise-pair}.
This concludes the proof.
\end{proof}
}

\section{Role of Distance Functions}
\label{sec:distance_fct}

In the previous section we have characterized the nonlinear eigenvalue problem \labelcref{eq:eigenvalue_prob_inf} which, in particular, is fulfilled by all minimizers of the Rayleigh quotient \labelcref{eq:rayleigh_inf}.
Now, we study the relations between general minimizers and the distance function, which is always a minimizer of the Rayleigh quotient but no infinity ground state or infinity harmonic potential, in general.

We first recall the well-known fact that the distance function is pointwise maximal among all minimizers of the Rayleigh quotient and show that its gradients are parallel to gradients of general minimizers, where the latter are maximal. 
Then, we construct a inner distance function, which is the distance function to a generalized inball of the domain $\Omega$, and show that it is quasi pointwise minimal among all minimizers.
The simple consequence is the known uniqueness of minimizers on stadium-like domains~\cite{yu2007some}, where the inner and the normal distance function coincide.

Very important for our following arguments is the fact that the distance function $d_\Omega$ is pointwise maximal among all minimizers of $R^\infty$.
For self-containedness we include the proof.
We also show that the high ridge $\highridge$, defined in \labelcref{eq:highridge}, where the distance function attains its maximum, contains the set of maximal points of any other minimizer of the Rayleigh quotient \labelcref{eq:rayleigh_inf}.
\begin{proposition}[Maximality of the distance function]\label{prop:maximality}
Let $u$ be a minimizer of $R^\infty$ with $\norm{\nabla u}_{\L^\infty}=1$, and let $d_\Omega$ denote the distance function of $\partial\Omega$. Then it holds that $\abs{u}\leq d_\Omega$
and
\begin{align}\label{eq:max_set_inclusion}
    \omegamax(u)\subset\omegamax(d_\Omega)=\highridge.
\end{align}
\end{proposition}
\begin{proof}
For $x\in\Omega$ we let $x_\Omega\in\argmin_{y\in\partial\Omega}|y-x|$ denote a projection onto the boundary.
Then using the Lipschitz continuity of $u$ it holds
\begin{align*}
    \abs{u(x)}=\abs{u(x)-u(x_\Omega)}\leq |x-x_\Omega|=d_\Omega(x),
\end{align*}
which proves the first claim.
Since both $u$ and $d_\Omega$ are minimizers it holds 
$$\norm{d_\Omega}_\infty=\norm{u}_\infty=\abs{u(x)}\leq d_\Omega(x)\leq\norm{d_\Omega}_\infty,\quad x\in\omegamax(u)$$
which proves \labelcref{eq:max_set_inclusion}.
\end{proof}
We make another observation: using \cref{prop:maximality} we can show show that gradients of minimizers of the Rayleigh quotient are parallel to gradients of the distance function on $\Omegamax(u)$.

\begin{proposition}[Parallelity of the gradients]\label{prop:parallelity_gradients}
Let $u,v\in\lip_0(\Omega)$ be non-negative minimizers of $R^\infty$ and assume that $\omegamax(v)=\highridge$ (e.g., $v$ could be the distance function \labelcref{eq:distance_fct}). Then it holds
\begin{align}
    {\nabla_{\tau} u\cdot\nabla_{\tau} v=|\nabla_{\tau} u|\,|\nabla_{\tau} v|,\quad\text{$\tau$-\text{a.e.} in $\Omega$},}
\end{align}
where $\tau\in\calM_+(\Omega)$ solves the optimality condition~\labelcref{eq:eigenfunction_div} for $u$.
\end{proposition}
\begin{proof}
Thanks to \cref{thm:main} there exist measure {$\nu,\tau$ with $\lambda_\infty\nu u=-\div(\tau\nabla_\tau u)$ where $\supp\nu\subset\omegamax(u)\subset\highridge=\omegamax(v)$.
Letting $\mu:=\nu u\in\Phi_{\C_0(\Omega)}(u)$ and $\sigma:=\tau\grad_\tau u\in\partial J^\infty(u)$, we get}
\begin{align*}
    J^\infty(v)
    &=\lambda_\infty\norm{v}_\infty
    =\lambda_\infty\int_\Omega v\d\mu 
    =-\int_\Omega v\d\div\sigma
    =\int_\Omega\nabla_{\abs{\sigma}}v\cdot\d\sigma\\
    &=\int_\Omega\nabla_{\abs{\sigma}}v\cdot\frac{\nabla_{\abs{\sigma}}u}{\norm{\nabla u}_{\L^\infty}}\d\abs{\sigma} 
    \leq \int_\Omega\abs{\nabla_{\abs{\sigma}}v}\frac{\abs{\nabla_{\abs{\sigma}}u}}{\norm{\nabla u}_{\L^\infty}}\d\abs{\sigma}\\
    &\leq J^\infty(v),
\end{align*}
where we used \cref{lem:norm-tangential-grad} for the last inequality.
Hence, $\nabla_{\abs{\sigma}}u\cdot\nabla_{\abs{\sigma}} v-|\nabla_{\abs{\sigma}} u||\nabla_{\abs{\sigma}} v|$ integrates to zero with respect to $\abs{\sigma}$, despite being non-positive. This is only possible, if the expression equals zero $\sigma$-\text{a.e.}
{Using $\abs{\sigma}=\tau\,J^\infty(u)$ yields the desired statement.}
\end{proof}

Next we study the role of another distance function which is essentially pointwise minimal among all minimizers of $R^\infty$.
To this end we define the generalized inball of $\Omega$ as
\begin{align}\label{eq:inball}
    \Omegain\defeq \{x\in\Omega\st\dist(x,\highridge)< \inradius\},
\end{align}
which is a ball if the high ridge $\highridge$ is a singleton and a stadium-like domain otherwise.
Now we define the distance function on $\Omegain$ extended by zero as
\begin{align}
    \din(x)=
    \begin{cases}
    \dist(x,\partial\Omegain),\quad&x\in\Omegain,\\
    0,\quad &x\in\Omega\setminus\Omegain,
    \end{cases}
\end{align}
and we refer to it as inner distance function.

\begin{proposition}
The inner distance function $\din$ fulfills the relation
\begin{align}\label{eq:formula_inner}
    \din(x) = \max(\inradius-\dist(x,\highridge),0),\quad x\in\Omega.
\end{align}
\end{proposition}
\begin{proof}
For $x\in\Omega\setminus\Omegain$ the identity is trivially true.
Let therefore $x\in\Omegain$ and we have to show that $\dist(x,\partial\Omegain)=\inradius-\dist(x,\highridge)$.
We can find two points $x_0\in\highridge$ and $\bar{x}\in\partial\Omegain$ such that $\dist(x,\highridge)=\abs{x-x_0}$ and $\dist(x,\partial\Omegain)=\abs{x-\bar{x}}$.
Using the triangle inequality, we can establish the inequality
\begin{align*}
    \dist(x,\partial\Omegain) + \dist(x,\highridge) 
    = \abs{x-\bar{x}} + \abs{x-x_0}
    \geq \abs{\bar{x}-x_0}\geq\inradius.
\end{align*}
Let us now consider the point $z_\lambda\defeq x+\lambda\frac{x-x_0}{\abs{x-x_0}}$.
One can easily compute that
\begin{align*}
    \abs{z_\lambda-x_0}=\abs{x-x_0}+\lambda
\end{align*}
and hence $z_\lambda\in\partial\Omegain$ iff $\lambda=\inradius-\abs{x-x_0}$.
For this value of $\lambda$ we can compute
\begin{align*}
    \dist(x,\partial\Omegain) + \dist(x,\highridge) \leq \abs{x-z_\lambda} + \abs{x-x_0}
    = \lambda + \abs{x-x_0} = \inradius.
\end{align*}
Hence, we have established both inequalities and showed \labelcref{eq:formula_inner}.
\end{proof}

Obviously $\din$ is also a minimizer of $R^\infty$ since it holds
\begin{align*}
    R^\infty(\din)=\frac{\norm{\nabla \din}_\infty}{\norm{\din}_\infty} = \frac{1}{\inradius} = \lambda_\infty.
\end{align*}
Notably, if $\highridge$ is not a singleton, $\din$ is not minimal among all minimizers of $R^\infty$ since cone-like functions with tips in $\highridge$ lie below $\din$. 
However, we have the following result:
\begin{proposition}[Quasi-minimality of the inner distance function]
If $\highridge$ is a singleton, then $\din\leq \abs{u}$ for all $u\in\argmin R^\infty$ with $\norm{\nabla u}_{\L^\infty}=1$.
In general, it holds that $\din\leq \abs{u}$ for all $u\in\argmin R^\infty$ that satisfy $\norm{\nabla u}_{\L^\infty}=1$ and $\argmax \abs{u} = \highridge$.
\end{proposition}
\begin{proof}
Since the first statement is a special case of the second one, we only prove the latter.
To this end assume that $\argmax \abs{u}=\highridge$.
Outside $\Omegain$ nothing needs to be shown so we assume that there is $x\in\Omegain$ with $\abs{u(x)}<\din(x)$. 
Letting $x_0\in\argmin_{y\in\highridge}|y-x|$ be a projection of $x$ onto the closed set $\highridge$, it holds using~\labelcref{eq:formula_inner}
\begin{align*}
    \abs{u(x_0)-u(x)} &\geq \abs{u(x_0)} - \abs{u(x)} \\
    &>\din(x_0)-\din(x) \\
    &= (\inradius-\dist(x_0,\highridge)-(\inradius-\dist(x,\highridge))\\
    &=\dist(x,\highridge)\\
    &=|x-x_0|,
\end{align*}
which again contradicts the fact that $u$ has unit Lipschitz constant.
\end{proof}

Since on stadium-like domains it holds $\Omegain=\Omega$ and hence $d_\Omega=\din$, one obtains uniqueness of minimizers which take their maximum on $\highridge$.
\begin{corollary}[Stadium-like domains]\label{cor:stadium}
Let $\Omega$ be a stadium-like domain, meaning $\Omega=\{x\in\R^n\st\dist(x,\highridge)<\inradius\}$. 
Then there is exactly one minimizer of $R^\infty$ which takes its maximum on $\highridge$ and it is given by the distance function $\dist(\cdot,\partial\Omega)$.
\end{corollary}

{
\begin{remark}[Sign-changing minimizers]\label{rem:sign-changing-minimizers}
If $\Omega$ is not a stadium-like domain, minimizers of $R^\infty$ exist, which change their sign. To construct such a minimizer $\tilde u$, one can set $\tilde u = \din$ on $\Omegain$ and extend it by a sign-changing function in $\Omega \setminus \Omegain$ in such a way that both $\norm{\tilde  u}_\infty = \norm{\din}_\infty$ and $\norm{\grad \tilde u}_\infty = \norm{\grad \din}_\infty$.
\end{remark}
}

\cref{cor:stadium} simplifies parts of the proof of uniqueness for infinity ground states on stadium-like domains from \cite{yu2007some}:
\begin{corollary}[Uniqueness of infinity ground states]
Let $\Omega$ be a stadium-like domain such that $\highridge$ is Lipschitz-connected.
Then any solution of \labelcref{eq:infty-groundstate} coincides with a multiple of the distance function.
\end{corollary}
\begin{proof}
In \cite{yu2007some} it was proved that every ground state takes its maximum on $\highridge$ if the high ridge is Lipschitz-connected. 
Hence, \cref{cor:stadium} immediately gives uniqueness.
\end{proof}

\begin{example}[Distance function of a ball]\label{ex:distance_function}
Let $\Omega=B_1(0)$ be the $n$-dimensional unit ball and $d_\Omega(x)=1-|x|$ be the distance function to $\partial\Omega$. 
Then $d_\Omega$ satisfies the eigenvalue problem \labelcref{eq:eigenfunction_div} with $\lambda_\infty=1$, {$\nu=\delta_0$}, and {$\tau$ being the} absolutely continuous measure {$\d\tau(x)=-\frac{1}{\omega_n\abs{x}^{n-1}}\d x$}. 
Here $\omega_n$ is the surface area of $B_1(0)$.

To see that this is true, we note that by construction $-\div(\tau\grad_\tau d_\Omega)\in\partial J^\infty(d_\Omega)$ as one can easily check.
It remains to check that it equals $\delta_0=\Phi_{\C_0(\Omega)}(d_\Omega)$ as measure. To this end, we take a smooth test function $\varphi\in \C_c^\infty(\Omega)$ with compact support and compute
\begin{align*}
    \langle{-\div(\tau\grad_\tau d_\Omega)},\varphi\rangle
    &=\int_{B_1(0)}\nabla \varphi(x) \cdot {\grad_\tau d_\Omega(x) \d\tau(x)} =-\int_{B_1(0)}\frac{1}{\omega_n}\frac{x}{|x|^{n}}\cdot\nabla\varphi\dx.
\end{align*}
At this point we observe that $x\mapsto -\frac{1}{\omega_n}\frac{x}{|x|^{n}}$ is precisely the gradient of the fundamental solution $\hat{u}$ to the Laplace equation on $\R^n$, given by
\begin{align*}
    \hat{u}(x)=
    \begin{cases}
        -\frac{1}{2\pi}\log|x|,\quad&n=2,\\
        \frac{1}{\omega_n(n-2)}\frac{1}{|x|^{n-2}},\quad&n>2.
    \end{cases}
\end{align*}
Since $\hat{u}$ solves $-\Delta\hat{u}=\delta_0$ in the sense of distributions, we obtain
\begin{align*}
    \langle{-\div(\tau\grad_\tau d_\Omega)},\varphi\rangle=\langle\nabla\hat{u},\nabla\varphi\rangle=\langle\delta_0,\varphi\rangle,\quad\forall \varphi\in \C_c ^\infty(\Omega),
\end{align*}
which concludes the proof.
\end{example}

In this example we had the situation that {a minimizer $u$ satisfies $\tau\grad_\tau u=\nabla\hat{u}$}, where $\hat{u}$ denotes the fundamental solution of the Laplace equation on the whole space.
However, this is certainly not the case on general domains.
On a square, for instance, the gradients of the distance function do not have the same direction as $\hat{u}$.
Still, also for general domains the {rotation-free} component of the measure {$\tau\grad_\tau u$} equals $\nabla\hat{u}$ as we see from the following more general example.
\begin{example}
Let $u$ be an arbitrary solution of the problem \labelcref{eq:eigenfunction_div} with $\lambda=\lambda_\infty$ and $\norm{\nabla u}_{\L^\infty}=1$.
According to \cref{thm:main} there is $\sigma\in \calD\calM(\Omega,\R^n)$ and $\mu\in\M(\Omega)$ supported in $\omegamax(u)\subset\highridge$ such that
\begin{align*}
    -\div\sigma=\lambda_\infty\mu.
\end{align*}
{Indeed, this holds for $\sigma:=\tau\grad_\tau u$ and $\mu:=\nu u$.}
It is known that the high ridge $\highridge$ has finite $\H^{n-1}$-measure \cite{kraft2016measure} and hence the same holds for $\supp\mu$.
We let $\hat{u}$ denote the solution of the problem 
\begin{align*}
    -\Delta\hat{u}=\frac{1}{\inradius\H^{n-1}(\supp\mu)}\H^{n-1}\restr\supp\mu,
\end{align*}
which can be constructed by integrating Green's function of the Laplacian along $\supp\mu$.

We can rewrite the optimality condition for $u$ as
\begin{align*}
    -\div\sigma=\frac{1}{\inradius\H^{n-1}(\supp\mu)}\H^{n-1}\restr_{\supp\mu}.
\end{align*}
Plugging in the fundamental solution we get
\begin{align*}
    -\div\sigma = -\Delta\hat u = -\div\nabla\hat{u}.
\end{align*}
Consequently, it holds
\begin{align*}
    \sigma=\nabla\hat{u}+\rho
\end{align*}
where $\rho\in\calD\calM(\Omega,\R^n)$ is a suitable divergence-measure field with $\div\rho=0$ such that $\abs{\sigma}=1$.
This decomposition of $\sigma$ is similar to what was shown in \cite{bungert2020structural} for a more regular situation, where $\sigma$ is restricted to be an $\L^2$ vector field.
\end{example}

\section{Relation to Optimal Transport}\label{sec:ot}

In light of the abstract \cref{prop:dual_quotient} we now investigate duality for minimizers of $R^\infty$ given by~\labelcref{eq:rayleigh_inf}.
Here we discover an interesting relationship to an optimal transport problem, generalizing the observations of \cite{champion2008infty} for infinity ground states to arbitrary minimizers of the Rayleigh quotient $\norm{\nabla u}_{\L^\infty}/\norm{u}_\infty$.

To this end we note that according to \cref{def:dual_fctl} the dual functional of $J^\infty$ is given by
\begin{align}
    J_*^\infty(\sg) = \sup\left\lbrace \int_\Omega u\d\mu \st u\in\lip_0({\Omega}),\; \norm{\nabla u}_{\L^\infty}\leq 1\right\rbrace,\quad \mu \in \calM(\Omega),
\end{align}
which is very similar to the Kantorovich--Rubinstein (KR) norm
\begin{align}\label{eq:kr_norm}
    \norm{\mu}_{\kr(\closure{\Omega})} = \sup\left\lbrace \int_\Omega u\d\mu \st u\in\lip(\closure{\Omega}),\; \norm{\nabla u}_{\L^\infty}\leq 1, \;\norm{u}_\infty\leq 1\right\rbrace,\quad\mu\in\calM(\closure{\Omega}).
\end{align}
The dual norm in $\C_0({\Omega})^*=\calM(\Omega)$ is given by the total variation $|\mu|(\Omega)$.
Hence, the dual quotient to $R^\infty$ is given by
\begin{align}\label{eq:dual_quotient_inf}
    R^\infty_*(\mu) = \frac{\abs{\mu}({\Omega})}{J^\infty_*(\mu)},\quad \mu \in \calM(\Omega)\setminus\{0\}.
\end{align}

Let us first understand the role of $J_*^\infty$ which can be interpreted both as an optimal transport distance to the boundary of the domain and as a norm on a quotient space. 
To see this, we first define the $1$-Wasserstein distance between two probability measures $\mu,\nu\in\calP(\closure{\Omega})$ as
\begin{align}
    W^1(\mu,\delta)=\sup\left\lbrace\int_\Omega u\d\mu - \int_\Omega u\d\delta\st u\in\lip(\closure{\Omega}),\; \norm{\nabla u}_{\L^\infty}\leq 1\right\rbrace.
\end{align}
Note that $\norm{\grad u}_{\L^\infty}$ equals the Lipschitz constant of $u$ with respect to the geodesic distance on $\Omega$ (see \cite[Eq. (1.6)]{bhattacharya1989limits}), which is why we call $W^1(\cdot,\cdot)$ a geodesic Wasserstein distance.
\begin{proposition}\label{prop:Wasserstein_projection}
For every $\mu\in\calP({\Omega})$ it holds 
\begin{align}
     J_*^\infty(\mu) = \inf_{\delta\in\calP(\partial\Omega)} W^1(\mu,\delta).
\end{align}
\end{proposition}
\begin{proof}
Since the set $\{u \in \lip(\closure{\Omega}) \colon \norm{\nabla u}_{\L^\infty}\leq 1\}$ is convex and weakly-* compact, the set $\{\delta \in \calP(\closure{\Omega}) \colon \supp(\delta) \subset \partial\Omega\}$ is convex and the objective function is weakly-* continuous, we apply the Nonsymmetrical Minmax Theorem from \cite[Th. 3.6.4]{Borwein_Zhu} and obtain
\begin{align*}
    \inf_{\delta\in\calP(\partial\Omega)} W^1(\mu,\delta) &= \sup_{\substack{u \in \lip(\closure{\Omega}) \\ \norm{\nabla u}_{\L^\infty}\leq 1}} \inf_{\delta\in\calP(\partial\Omega)}\int_\Omega u\d\mu - \int_\Omega u\d\delta \\
    &=\sup_{\substack{u \in \lip(\closure{\Omega}) \\ \norm{\nabla u}_{\L^\infty}\leq 1}} \int_\Omega u\d\mu - \sup_{\delta\in\calP(\partial\Omega)}\int_\Omega u\d\delta \\
    &= \sup_{\substack{u \in \lip(\closure{\Omega}) \\ \norm{\nabla u}_{\L^\infty}\leq 1}} \int_\Omega u\d\mu - \sup_{x \in \partial \Omega} \abs{u(x)} \\
    &= \sup_{\substack{u \in \lip(\closure{\Omega}) \\ \norm{\nabla u}_{\L^\infty}\leq 1}} \int_\Omega (u-M_u)\d\mu,
\end{align*}
where we abbreviated $M_u\defeq \sup_{x\in\partial\Omega}|u(x)|$ and used the fact that $\mu$ is a probability measure.
Substituting $u \leftrightarrow u-M_u$ in the supremum, we get 
\begin{align*}
    \inf_{\delta\in\calP(\partial\Omega)} W^1(\mu,\delta) 
    = \sup_{\substack{u\in\lip(\closure{\Omega}) \\ \norm{\nabla u}_{\L^\infty} \leq 1 \\ u\vert_{\partial\Omega}\leq 0}} \int_\Omega u \d\mu
    = \sup_{\substack{u\in\lip_0({\Omega}) \\ \norm{\nabla u}_{\L^\infty} \leq 1}} \int_\Omega u \d\mu 
    = J_*^\infty(\mu).
\end{align*}
Here we used the non-negativity of $\mu$ to conclude that the constraint $u\vert_{\partial\Omega}\leq 0$ can be converted to $u\in\lip_0({\Omega})$.
\end{proof}

For the interpretation of $J_*^\infty$ as a norm, we let $\kr(\closure{\Omega})$ denote the KR-space, i.e., the completion of $\calM(\closure\Omega)$ with respect to the norm~\labelcref{eq:kr_norm}.
The dual of $\kr(\closure{\Omega})$ coincides with $\lip(\closure{\Omega})$.
The space $\kr(\partial\Omega)$, being defined in an analogous manner, can be identified with a closed subspace of $\kr(\closure{\Omega})$ of measures that are zero on $\Omega \setminus \partial\Omega$, which allows us to consider the quotient space
\begin{align}\label{eq:def_kr_delta}
    \kr_\partial(\Omega) \defeq  \kr(\closure{\Omega})/\kr(\partial\Omega),
\end{align}
where the equivalence relation is 
\begin{align}
    \mu \sim \nu \iff \mu - \nu\in \kr(\partial\Omega).
\end{align}
With our notation we already indicate that $\kr_\partial(\Omega)$ depends only on $\Omega$ and not its closure.
The canonical norm on $\kr_\partial(\Omega)$ is given by 
\begin{align}
    \norm{\mu}_{\kr_\partial(\Omega)}\defeq \inf_{\nu\in \kr(\partial\Omega)}\norm{\mu-\nu}_{\kr(\closure{\Omega})}
\end{align}
and we have the following result:
\begin{proposition}\label{prop:dual_KR_delta}
The dual of $\kr_\partial(\Omega)$ is given by $\lip_0({\Omega})$, i.e.,
\begin{align}
    \left(\kr_\partial(\Omega)\right)^* = \lip_0({\Omega}).
\end{align}
\end{proposition}
\begin{proof}
Since $\kr_\partial(\Omega)$ defined in~\labelcref{eq:def_kr_delta} is a quotient, its dual space coincides with the annihilator of $\kr(\partial\Omega)$ in $\left(\kr(\closure{\Omega})\right)^*=\lip(\closure{\Omega})$, which is given by $\lip_0({\Omega})$.
\end{proof}

\begin{proposition}
$J_*^\infty(\mu)$ is an equivalent norm on $\kr_\partial(\Omega)$. Moreover, if $r_\Omega \leq 1$, then for any $\mu\in\calM(\Omega)$ it holds that
\begin{align}
    J_*^\infty(\mu) = \norm{\mu}_{\kr_\partial(\Omega)}.
\end{align}
\end{proposition}
\begin{proof}
Owing to \cref{prop:dual_KR_delta} we can express the norm on $\kr_\partial(\Omega)$ by duality as follows 
\begin{align*}
    \norm{\mu}_{\kr_\partial(\Omega)} = \sup_{\substack{u\in\lip_0({\Omega}) \\ \norm{\nabla u}_{\L^\infty}\leq 1 \\ \norm{u}_\infty \leq 1}} \int_\Omega u \d\mu \leq \sup_{\substack{u\in\lip_0({\Omega}) \\ \norm{\nabla u}_{\L^\infty}\leq 1}} \int_\Omega u \d\mu = J_*^\infty(\mu).
\end{align*}
It is obvious that for any $u\in\lip_0({\Omega})$ with $\norm{\nabla u}_{\L^\infty}\leq 1$ it holds that $\norm{u}_\infty\leq r_\Omega$. Let $t \defeq  \max\{1,r_\Omega\}$. Then we have
\begin{align*}
    J_*^\infty(\mu) &= \sup_{\substack{u\in\lip_0({\Omega}) \\ \norm{\nabla u}_{\L^\infty}\leq 1}} \int_\Omega u \d\mu = t \sup_{\substack{u\in\lip_0({\Omega}) \\ \norm{\nabla u}_{\L^\infty}\leq 1}} \int_\Omega \frac{1}{t} u \d\mu 
    = t \sup_{\substack{u\in\lip_0({\Omega}) \\ \norm{\nabla u}_{\L^\infty}\leq 1 \\ \norm{u}_\infty \leq t}} \int_\Omega \frac{1}{t} u \d\mu \\
    &= t \sup_{\substack{u\in\lip_0({\Omega}) \\ \norm{\nabla u}_{\L^\infty}\leq \frac{1}{t} \\ \norm{u}_\infty \leq 1}} \int_\Omega  u \d\mu 
    \leq t \sup_{\substack{u\in\lip_0({\Omega}) \\ \norm{\nabla u}_{\L^\infty}\leq 1 \\ \norm{u}_\infty \leq 1}} \int_\Omega  u \d\mu 
    = t \norm{\mu}_{\kr_\partial(\Omega)},
\end{align*}
hence the equivalence. If $r_\Omega \leq 1$, we get that $t=1$ and
\begin{align*}
\norm{\mu}_{\kr_\partial(\Omega)}=J_*^\infty(\mu),
\end{align*}
which completes the proof.
\end{proof}
\begin{remark}
If $r_\Omega \geq 1$, we can define equivalent $\kr$ norms as follows (cf.~\labelcref{eq:kr_norm})
\begin{align*}
    \norm{\mu}_{\kr(\closure{\Omega})} &= \sup\left\lbrace \int_\Omega u\d\mu \st u\in\lip(\closure{\Omega}),\; \norm{\nabla u}_{\L^\infty}\leq 1, \;\norm{u}_\infty\leq r_\Omega \right\rbrace,\quad\mu\in\calM(\Omega), \\
    \norm{\mu}_{\kr(\closure{\Omega})} &= \sup\left\lbrace \int_\Omega u\d\mu \st u\in\lip(\closure{\Omega}),\; \norm{\nabla u}_{\L^\infty}\leq 1, \;\norm{u\vert_{\partial\Omega}}_\infty\leq 1 \right\rbrace,\quad\mu\in\calM(\Omega).
\end{align*}
In both cases we get that $\norm{\mu}_{\kr_\partial(\Omega)}=J_*^\infty(\mu)$ regardless of the value of $r_\Omega$.
\end{remark}

Analysing the minimizers of $R_*^\infty$ is fairly easy since they can be explicitly computed.

\begin{proposition}\label{prop:minimizers_dual_rayleigh}
The minimizers of $R^\infty_*$ are given by all non-zero multiples of $\mu\in\calP(\closure{\Omega})$ with $\supp(\mu)\subset\highridge$.
\end{proposition}
\begin{proof}
Since the minimization of $R^\infty_*$ given by~\labelcref{eq:dual_quotient_inf} is homogeneous, the problem is equivalent to the maximization of $J_*^\infty(\mu)$ among all $\mu\in\calM(\Omega)$ with $\abs{\mu}({\Omega})=1$.
Since $J_*^\infty(\mu)\leq J_*^\infty(\abs{\mu})$, we can further restrict ourselves to $\mu\in\calP({\Omega})$.
Next, for any $\mu\in\calP({\Omega})$ it holds
$$J_*^\infty(\mu)\leq \sup\left\lbrace \norm{u}_\infty\ \st u\in\lip_0(\closure{\Omega}),\;\norm{\nabla u}_{\L^\infty} \leq 1\right\rbrace\leq r_\Omega.$$
If additionally $\supp(\mu)\subset\highridge$, one obtains
\begin{align*}
    J_*^\infty(\mu)\geq \int_{\highridge}d_\Omega\d\mu=r_\Omega,
\end{align*}
which proves the assertion.
\end{proof}

{
It remains to deduce \labelcref{eq:nu_OT_problem} from \cref{prop:minimizers_dual_rayleigh}.
\begin{proposition}\label{prop:OT_problem}
Let $u\in\lip_0(\Omega)$ be a minimizer of the Rayleigh quotient $R^\infty$ and let $\nu$ and $\tau$ as in \cref{thm:main} such that $\lambda_\infty\nu u = -\div(\tau\grad_\tau u)$.
Then the probability measure $\mu := \nu \norm{u}_\infty$ solves \labelcref{eq:nu_OT_problem}, i.e.,
\begin{align*}
    \mu \in \argmax_{\tilde\mu\in\calP(\Omega)}\inf_{\delta\in\calP(\partial\Omega)}W^1(\tilde\mu,\delta).
\end{align*}
\end{proposition}
\begin{proof}
The measure $\mu:=\nu\norm{u}_\infty$ is a probability measure and satisfies $\supp(\mu)\subset\omegamax(u)\subset\highridge$.
Hence, \cref{prop:minimizers_dual_rayleigh} implies that it minimizes $R^\infty_*$ which is equivalent to maximizing $J^\infty_*$ over all probability measures.
The reformulation of $J^\infty_*$ in \cref{prop:Wasserstein_projection} concludes the proof.
\end{proof}
}

\section{Summary and Outlook}
\label{sec:outlook}

In this paper we have characterized the subdifferentials of the functionals $u\mapsto\norm{u}_\infty$ and $u\mapsto\norm{\nabla u}_{\L^\infty}$ over the Banach space $\C_0(\Omega)$ in order to characterize the nonlinear eigenvalue problem associated to the Rayleigh quotient $\frac{\norm{\nabla u}_{\L^\infty}}{\norm{u}_\infty}$.
For this we solely relied on duality between continuous functions and Radon measures and utilized the concept of tangential gradients. 
Our results show that general stationary points of the Rayleigh quotient satisfy a fully nonlinear PDE in divergence form.
We also studied geometric properties of minimizers and related them to the inner distance function and the distance function to the boundary of $\Omega$.
Finally, we showed that minimization of the Rayleigh quotient is equivalent to an optimal transport problem involving a generalized Kantorovich--Rubinstein norm.
We derived this equivalence using a dual Rayleigh quotient which is defined on the space of measures and whose minimizers are subgradients of primal minimizers.

Some open questions which are subject of future work are the following ones:

First, we would like to investigate whether and how the concept of dual Rayleigh quotients, which we have introduced in this paper, can be utilized for studying approximation with finite $p$.
For instance, infinity ground states \labelcref{eq:infty-groundstate} arise as limit of $p$-Laplacian eigenfunctions for $p\in(1,\infty)$ which solve
\begin{align*}
    \min_{u\in\W^{1,p}_0(\Omega)}\frac{\norm{\nabla u}_{\L^p}}{\norm{u}_{\L^p}}
\end{align*}
or equivalently
\begin{align*}
    \lambda_p \abs{u}^{p-2}u = -\Delta_p u,
\end{align*}
{where $\Delta_p u := \div(\abs{\grad u}^{p-2}\grad u)$.}
The minimization problem of the dual Rayleigh quotient in this case is given by 
\begin{align*}
    \min_{\mu\in\L^q(\Omega)} \frac{\norm{\mu}_{\L^q}}{\norm{\mu}_{\W^{-1,q}_0}},
\end{align*}
where $q\in(1,\infty)$ is the conjugate exponent to $p$ such that $1/p+1/q=1$ and the negative Sobolev space $\W^{-1,q}_0$ is the dual of $\W^{1,p}_0$.
Using subdifferentials, the optimality conditions of this minimization problem can be computed and are given by
\begin{align*}
    \lambda_p \mu = -\Delta_p (|\mu|^{q-2}\mu).
\end{align*}
Indeed, this PDE can be seen to be equivalent to the $p$-Laplacian eigenvalue problem via the identification $\mu=\abs{u}^{p-2}u$.
Sending $p\to\infty$ (i.e., $q\to 1$) solutions $\mu_q$ of this PDE should converge to a measure $\mu$ which is a subgradient of the corresponding infinity ground state, i.e., $\mu\in\partial J^\infty(u)$.
Correspondingly, the vector fields $|\nabla(\abs{\mu_q}^{q-2}\mu_q)|^{p-2}\nabla(\abs{\mu_q}^{q-2}\mu_q)$ should converge to a measure $\sigma$ which satisfies $\lambda_\infty\mu=-\div\sigma$ and has the properties of \cref{thm:main}.
We suppose that this limit $\sigma$ admits some minimality properties (for instance related to its support) compared to arbitrary calibrations whose divergence is a subgradient. 

Second, we would like to apply our subdifferential calculus and optimal transport interpretation to Lipschitz extensions of a Lipschitz function $g:\partial\Omega\to\R$, i.e., solutions of
\begin{align}\label{eq:lip_ext}
    \min\left\lbrace\norm{\nabla u}_{\L^\infty} \st u\in \W^{1,\infty}(\Omega),\;u=g \text{ on }\partial\Omega\right\rbrace.
\end{align}
Absolute minimizers satisfy the infinity Laplacian equation
\begin{align*}
    \begin{cases}
        \Delta_\infty u = 0,\quad&\text{in }\Omega,\\
        u = g,\quad&\text{on }\partial\Omega.
    \end{cases}
\end{align*}
For this equation it was already shown in \cite{evans2005various} that the solution solves the divergence PDE $\div(\nu\nabla u)=0$ and a similar statement for sure can be proved for general minimizers of \labelcref{eq:lip_ext}.

\section*{Acknowledgements}
This work was supported by the European Unions Horizon 2020 research and innovation programme under the Marie Sk{\l}odowska-Curie grant agreement No 777826 (NoMADS).
LB acknowledges funding by the Deutsche Forschungsgemeinschaft (DFG, German Research Foundation) under Germany's Excellence Strategy - GZ 2047/1, Projekt-ID 390685813.
YK acknowledges financial support of the EPSRC (Fellowship EP/V003615/1), the Cantab Capital Institute for the Mathematics of Information at the University of Cambridge and the National Physical Laboratory.

\printbibliography

\end{document}